\documentclass[12pt]{IEEEtran}

\usepackage{amsmath}
\usepackage{amsthm}
\usepackage{algorithm}
\usepackage{tikz}
\usepackage{graphicx}
\usepackage{authblk}
\usepackage{mathtools}
\usepackage{tikz}
\usetikzlibrary{arrows,automata}
\usepackage{amsfonts}
\usepackage[noend]{algpseudocode}
\usepackage[noadjust]{cite}
\usepackage{multirow}
\usepackage{balance}
\usepackage{url}
\usepackage{colortbl}

\ifCLASSOPTIONcompsoc
\usepackage[caption=false,font=normalsize,labelfont=sf,textfont=sf]{subfig}
\else
\usepackage[caption=false,font=footnotesize]{subfig}
\fi

\title{Linear Encodings for Polytope Containment Problems}
\author{Sadra Sadraddini, ~\IEEEmembership{Member}, and Russ Tedrake, ~\IEEEmembership{Member} \thanks{
The authors are with Computer Science and Artificial Intelligence Laboratory (CSAIL) at Massachusetts Institute of Technology, 32 Vassar st, Cambridge, MA 02139,
\texttt{\{sadra,russt\}@mit.edu}}}

\theoremstyle{remark}
\newtheorem{define}{Definition}
\newtheorem{example}{Example}

\newtheorem{lemma}{Lemma}
\newtheorem{remark}{Remark}

\newtheorem{theorem}{Theorem}
\newtheorem{corollary}{Corollary}
\newtheorem{proposition}{Proposition}

\DeclareMathOperator{\convh}{{Convexhull}}

\DeclareMathOperator{\st}{{such~ that~}}

\DeclareMathOperator{\blk}{{blk}}

\DeclareMathOperator{\XX}{{\mathbb{X}}}
\DeclareMathOperator{\YY}{{\mathbb{Y}}}

\DeclareMathOperator{\range}{{\text{range}}}
\DeclareMathOperator{\kernel}{{\text{ker}}}

\DeclareMathOperator{\proj}{{\text{proj}}}
\DeclareMathOperator{\ball}{{\text{Ball}}}
\DeclareMathOperator{\reduced}{{\text{Redu.}}}

\begin{document}
\maketitle

\begin{abstract}
The polytope containment problem is deciding whether a polytope is a contained within another polytope. This problem is rooted in computational convexity, and arises in applications such as verification and control of dynamical systems. The complexity heavily depends on how the polytopes are represented. Describing polytopes by their hyperplanes (H-polytopes) is a popular representation. In many applications we use affine transformations of H-polytopes, which we refer to as AH-polytopes. Zonotopes, orthogonal projections of H-polytopes, and convex hulls/Minkowski sums of multiple H-polytopes can be efficiently represented as AH-polytopes. 
While there exists efficient necessary and sufficient conditions for AH-polytope in H-polytope containment, the case of AH-polytope in AH-polytope is known to be NP-complete. In this paper, we provide a sufficient condition for this problem that is cast as a linear program with size that grows linearly with the number of hyperplanes of each polytope. Special cases on zonotopes, Minkowski sums, convex hulls, and disjunctions of H-polytopes are studied. These efficient encodings enable us to designate certain components of polytopes as decision variables, and incorporate them into a convex optimization problem. We present examples on the zonotope containment problem, polytopic Hausdorff distances, zonotope order reduction, inner approximations of orthogonal projections, and demonstrate the usefulness of our results on formal controller verification and synthesis for hybrid systems.         
\end{abstract}

\begin{IEEEkeywords}
Polytope Containment Problem, Zonotopes, Formal Synthesis and Verification
\end{IEEEkeywords}

\thispagestyle{empty}
\pagestyle{empty}

\section{Introduction}

\IEEEPARstart{W}{e} are interested in establishing the conditions for the following relation to hold:
\begin{equation}
\label{eq_ana}
\mathbb{X} \subseteq \mathbb{Y},
\end{equation}
where $\mathbb{X}, \mathbb{Y} \subset \mathbb{R}^n$ are polytopes. We call \eqref{eq_ana} the \emph{polytope containment problem}, which is a subfamily of set containment problems (SCPs) \cite{freund1985complexity,mangasarian2002set}. We refer to $\mathbb{X}$ in \eqref{eq_ana} as the \emph{inbody}, and $\mathbb{Y}$ as the \emph{circumbody}. The decision problem \eqref{eq_ana} appears in applications such as computational geometry \cite{ghosh1990solution}, machine learning \cite{fung2003knowledge}, and control theory \cite{han2016enlarging}. For example, the inbody can represent the reachable states of a dynamical system, while the circumbody is the target set of states. In this context, the reachability verification problem becomes a polytope containment problem. When the inbody is characterized by the parameters of a controller, synthesizing the controller requires finding the parameters subject to \eqref{eq_ana}. The focus of this paper is not only providing a Boolean answer to \eqref{eq_ana}, but finding an efficient linear encoding, so \eqref{eq_ana} can be added to the constraints of a (potentially mixed-integer) linear/quadratic program. 

The complexity of writing \eqref{eq_ana} as linear constraints heavily depends on how the inbody and the circumbody are represented. In general, there are two fundamental ways to represent a polytope: representation by hyperplanes (H-polytope), or representation by vertices (V-polytope). H-polytopes are almost always preferred in high dimensions as the number of vertices is often very large. For example, a box in $n$ dimensions has $2n$ hyperplanes, but $2^n$ vertices. There exist several algorithms for conversions between H-polytopes and V-polytopes \cite{avis1992pivoting,bremner1998primal}, but their worst-case complexities are exponential in the number of dimensions, making it often impractical to navigate between H-polytopes and V-polytopes beyond 2 or 3 dimensions. Therefore, we do not focus on V-polytopes in this paper. 
\subsection{Problem Statement}

We consider the following generic form of \eqref{eq_ana}, where
\begin{equation}
\label{eq_ata}
\mathbb{X}=\bar{x} + X \mathbb{P}_x, \mathbb{Y}= \bar{y} + Y \mathbb{P}_y,
\end{equation}
where $\bar{x}, \bar{y} \in \mathbb{R}^n$, ${X} \in \mathbb{R}^{n \times n_x}, Y \in \mathbb{R}^{n \times n_y}$, and $\mathbb{P}_x \subset \mathbb{R}^{n_x}, \mathbb{P}_y \subset \mathbb{R}^{n_y}$ are given H-polytopes. In other words, $\mathbb{X}$ and $\mathbb{Y}$ are affine transformations of H-polytopes, which we refer to as AH-polytopes. AH-polytopes are very expressive. For example, one can write Minkowski sums and convex hulls of multiple H-polytopes as an AH-polytope. Zonotopes, which are widely used in estimation and control theory \cite{girard2005reachability,alamo2005guaranteed}, are affine transformations of boxes. While one can use a quantifier elimination method, such as the Fourier-Motzkin elimination method \cite{dantzig1972fourier}, to find the H-polytope form of an AH-polytope, it may lead to an exponential number of hyperplanes. We desire to cast \eqref{eq_ana} using \eqref{eq_ata} without explicitly computing H-polytope forms of $\XX$ or $\YY$.  

\subsection{Main Contributions}
Using convex analysis \cite{grunbaum1967convex,rockafellar2015convex}, we provide a set of linear constraints that provide sufficient conditions for \eqref{eq_ana}. We show that the conditions become necessary when $Y$ in \eqref{eq_ana} has a left-inverse. We provide interpretations for the conservativeness of our sufficient conditions. In our examples, we empirically find that the conservativeness is very small for zonotope containment problems. We also provide necessary and sufficient conditions when disjunctions of multiple instances of \eqref{eq_ana} for the same $\mathbb{X}$ and different H-polytope circumbodies are given. We show applications of our results on several problems for which existing methods are not satisfactory. We provide alternative methods for computing Hausdorff distance between polytopes, zonotope order reduction \cite{yang2018comparison}, inner approximation of orthogonal projections, and built upon our prior work in \cite{sadraddini2018sampling} to compute explicit model predictive control (MPC) schemes for constrained hybrid systems. 
\subsubsection*{Software}
The scripts of our results are available in a python package called \texttt{pypolycontain} \footnote{Available in \url{github.com/sadraddini/pypolycontain}}. It consists of a library of practically useful versions of $\eqref{eq_ana}$ encoded as (mixed-integer) linear constraints so they can be added to an optimization problem using Gurobi \cite{gurobi}.

\subsection{Related Work}
The polytope containment problem drew attention in the mathematics and optimization literature decades ago. Containment problems for convex sets are closely related to their dual characterization, support functions, and polar sets \cite{goberna2006dual,rockafellar2015convex}. Containment is also closely related to the S-lemma (or S-procedure) \cite{jonsson2001lecture}, which is often used for implications of quadratic inequalities. Freund and Orlin  \cite{freund1985complexity} studied the polytope containment problem for H-polytopes and V-polytopes, and proved that deciding if a H-polytope is contained within a V-polytope is co-NP-complete. The remaining three combinations (H-polytope inside H-polytope, V-polytope inside H-polytope, and V-polytope inside V-polytope) can be decided in polynomial time. 
Grittzman and Klee \cite{gritzmann1994complexity} studied a broad range of polytopic containment problems, with a particular focus on finding the inner and outer radius of H-polytopes and V-polytopes, as well as parallelotopes (a limited family of zonotopes). However, they did not report results on polytopic containment problems for zonotopes and other affine transformations of H-polytopes. The authors in \cite{kellner2013containment} extended the work in \cite{freund1985complexity} with 5 new cases where the inbody or circumbody could be spectrahedra - convex bodies that characterize the feasible set of semidefinite programs. Tiwary \cite{tiwary2008hardness} showed that the deciding whether an H-polytope is equivalent to the convex hull or the Minkowski sum of two H-polytopes is NP-complete. The proof relied on the fact that such a decision problem must involve vertex/facet enumeration. More recently, Kellner \cite{kellner2015containment} proved that polytope containment problem for projections of H-polytopes is also NP-complete, and cast \eqref{eq_ata} as a bilinear optimization problem, which can be solved using sequential semidefinite programs. This paper provides linear sufficient conditions for this problem. In the controls literature, the polytope containment problem is relevant to proving correctness of the controllers, and various versions of \eqref{eq_ata} have been applied for designing robust control invariant sets and tube model predictive controllers \cite{rakovic2007optimized,rakovic2010parameterized}. The authors in \cite{han2016enlarging} proposed a sufficient condition for a special case of zonotope containment problems, which were used to compute backward reachable sets of dynamical systems. Here we show that for the special case in \cite{han2016enlarging}, necessary and sufficient conditions actually exist, while we provide a sufficient result for general zonotopes. 

\subsection{Organization}
This paper is organized as follows. In Section \ref{sec_prelim}, we provide the necessary notation and formalize the terms used in this paper. We provide the main result in Section \ref{sec_main}. Next, we focus on special cases of Zonotopes, Minkowski sums, convex hulls, and disjunctive containment problems in Section \ref{sec_special}. Examples are provided throughout the paper, as well as a case study on explicit model predictive control in Section \ref{sec_control}.

\section{Preliminaries}
\label{sec_prelim}
The set of real and non-negative real numbers are denoted by $\mathbb{R}$ and $\mathbb{R}_+$, respectively. 
Given matrices $A_1,A_2$ of appropriate dimensions, we use $[A_1,A_2]$, $(A_1,A_2)$, and $\blk(A_1,A_2)$ to denote the matrices obtained by stacking $A_1$ and $A_2$ vertically, horizontally, and block-diagonally, respectively. 
 Given $\mathbb{S} \subset \mathbb{R}^n$ and $A {\in} \mathbb{R}^{n_A \times n}$, we interpret $A\mathbb{S}$ as $\{As | s \in \mathbb{S}\}$. Given two sets $\mathbb{S}_1, \mathbb{S}_2 \subset \mathbb{R}^n$, their Minkowski sum is denoted by $\mathbb{S}_1 \oplus \mathbb{S}_2 = \{s_1+s_2 | s_1 \in \mathbb{S}_1, s_2 \in \mathbb{S}_2\}$. Given $s \in \mathbb{R}^n$, $s+\mathbb{S}$ is interpreted as $\{s\} \oplus \mathbb{S}$. 
Given sets $\mathbb{S}_i \subset \mathbb{R}^n,i=1,\cdots,N$, their convex hull is defined as:
\begin{equation}
\convh(\{\mathbb{S}_i\}_{i=1,\cdots,N})=\bigoplus_{i=1}^N \lambda_i \mathbb{S}_i,
\end{equation}
where $\lambda_i \ge 0, i=1,\cdots,N$ and $\sum_{i=1}^N \lambda_i=1$.

Given matrix $A$, we use $\range(A)$ and {$\kernel(A)$ to denote its column-space and null-space, and $A'$ and $A^\dagger$ to denote its transpose and Moore-Penrose inverse, respectively. The matrix $|A|$ is the matrix obtained by taking the absolute values of $A$, element-wise. The infinity norm of matrix $A$ is denoted by $\left\|A\right\|_\infty$, which is the maximum absolute row sum. The identity matrix and the vector of all ones are denoted by $I$ and $\underbar{1}$, where the dimension is unambiguously interpretable from the context. All matrix inequality relations are interpreted element-wise in this paper.

A \emph{polyhedron} $\mathbb{P} \subset \mathbb{R}^n$ is a set that can be represented as the intersection of a finite number of closed half-spaces in the form $\mathbb{P}=\{ x \in \mathbb{R}^n | H x \le h\}$, where $H \in \mathbb{R}^{n_H \times n}, h \in \mathbb{R}^{n_H}$ define the hyperplanes.  A bounded polyhedron is called a \emph{polytope}.
Polytopes are closed under affine transformations, Minkowski sums, convex hulls, and intersections.

\begin{define}[H-Polytope]\cite{ziegler2012lectures}
An H-polytope $\mathbb{P}$ is a polytope which its hyperplanes are available: $\mathbb{P}=\{x \in \mathbb{R}^n | H x \le h\}$, $\kernel(H)=\{0\}$, where $n$ is the dimension of the polytope and $H,h$ are matrices defining the hyperplanes with appropriate dimensions. 
\end{define}

\begin{define}[AH-polytope]
An AH-polytope $\mathbb{X} \subset \mathbb{R}^{n}$ is a polytope that is given as an affine transformation of an H-polytope $\mathbb{P} \subset \mathbb{R}^m$, $\mathbb{X}=\bar{x}+X \mathbb{P}$, where $X \in \mathbb{R}^{n \times m}, \bar{x} \in \mathbb{R}^n$.
\end{define}
As mentioned earlier, converting an AH-polytope to its equivalent H-polytope may have an exponential complexity. The special case in which conversion is simple is when $X$ has a left inverse, in which case we have:
\begin{equation} 
\begin{array}{ll}
\{ \bar{x}+Xx | Hx \le h\} & \\
= \{ y \in \mathbb{R}^n | HX^{\dagger} y \le h + HX^{\dagger}\bar{x}\}, & X^\dagger X=I.
\end{array}
\end{equation}

Given two H-polytopes $\mathbb{P}_i=\{x \in \mathbb{R}^{n} | H_i x \le h_i\}, i=1,2$, their intersection can be easily represented as the following H-polytope: $
\mathbb{P}_1 \cap \mathbb{P}_2=\{x \in \mathbb{R}^{n} | [H_1,H_2] x \le [h_1,h_2]\}.
$
However, the H-polytope form of $\mathbb{P}_1 \oplus \mathbb{P}_2$ is not easy to obtain.  
Unlike H-polytopes, AH-polytopes are suitable for affine transformations and Minkowski sums, while the case of intersections is less trivial but still possible. Let $\mathbb{X}_i=\bar{x}_i+X_i\mathbb{P}_i, \mathbb{P}_i=\{z \in \mathbb{R}^{n_i} | H_i z \le h_i\}, i=1,2, X_i \in \mathbb{R}^{n \times {n_i}}, \bar{x}_i \in \mathbb{R}^n$, be two AH-polytopes. Note that, similar to H-polytopes, multiple AH-polytopes can represent the same polytope. 
\begin{itemize}
\item (Affine transformation) Given $g \in \mathbb{R}^q, G \in \mathbb{R}^{q \times n}$, we have:
\begin{equation}
G(\bar{x}+X\mathbb{P})+g=(G\bar{x}+g)+GX\mathbb{P}.
\end{equation}

\item (Minkowski Sum) We have the following relation:
\begin{equation}
(\bar{x}_1+X_1\mathbb{P}_1) \oplus (\bar{x}_2+X_2\mathbb{P}_2) = \bar{x}_1+\bar{x}_2+ (X_1,X_2)\mathbb{P}_{\oplus},
\end{equation}
where 
$$
\mathbb{P}_{\oplus}=\{ z \in \mathbb{R}^{n_1+n_2} | \blk(H_1,H_2)z \le [h_1,h_2] \}.
$$

\item (Intersection) from $\bar{x}_1+X_1 p_1=\bar{x}_2+X_2 p_2$ we obtain $p_2=X_2^\dagger (X_1 p_1 + \bar{x}_1-\bar{x}_2)+(I-X_2^\dagger X_2)w, w \in \mathbb{R}^{n_2}$. Therefore, we have the following relation:
\begin{equation}
(\bar{x}_1+X_1\mathbb{P}_1) \cap (\bar{x}_2+X_2\mathbb{P}_2) = (X_1,0) \mathbb{P}_{\cap},
\end{equation}
where  
$$
\begin{array}{ll}\mathbb{P}_{\cap}=& \{ z \in \mathbb{R}^{n_1+n_2} |[(H_1,0),(H_2X_2^\dagger X_1, I-X_2^\dagger X_2)]z\\
& \le [h_1,h_2+X_2^\dagger(\bar{x}_2-\bar{x}_1)] \}.
\end{array}
$$

\end{itemize}

\begin{define}[Unit Box]
The $n$-dimensional unit box, or the unit ball corresponding to $L_\infty$ norm, denoted by $\mathbb{B}_n$, is defined as $\mathbb{B}_n:=\{ x\in \mathbb{R}^n | \left|x\right\|_\infty \le 1\}$. Its H-polytope form is $\mathbb{B}_n= \{x \in \mathbb{R}^n | [I,-I] x \le \underbar{1}\}$. 
\end{define}
\begin{define}[Full-Dimensional Polytope]
A polytope $\mathbb{P} \subset \mathbb{R}^n$ is \emph{full-dimensional} if there exists $\bar{x} \in \mathbb{P}$, $\epsilon>0$, such that $\bar{x}+\epsilon \mathbb{B}_n \subset \mathbb{P}$. 
\end{define}

\begin{define}[Zonotope]
A \emph{zonotope} $\mathbb{Z}$ is a polytope that can be written as an affine transformation of the unit box $\mathbb{Z}:= \langle \bar{x},X \rangle = \bar{x} + X \mathbb{B}_{n_x}$, where $\bar{x} \in \mathbb{R}^n$ is the center and $X \in \mathbb{R}^{n \times n_x}$ is the generator matrix. The zonotope order is defined as $\frac{n_x}{n}$.
\end{define}
Zonotopes are a special case of AH-polytopes. An appealing feature of zonotopes is its operational convenience with Minkowski sums:
\begin{equation*}
\langle \bar{x}_1,X_1 \rangle \oplus \langle \bar{x}_2,X_2 \rangle = \langle \bar{x}_1 + \bar{x}_2, (X_1,X_2) \rangle.
\end{equation*} 
In practice, most zonotopes have order greater than one. Finding the H-polytope version of a zonotope requires facet enumeration, which its worst-case complexity is exponential in $n$ and ${n_x}$ \cite{althoff2015computing},  the number of rows and columns of the generator, respectively. 

\section{Main Results}
\label{sec_main}
In this section, we provide the main result of this paper in Theorem \ref{thm_main}. First, we restate the well-known result on H-polytope in H-polytope containment.  

\begin{lemma}[H-Polytope in H-Polytope]
\label{lemma_HH}
Let $\mathbb{X}=\{x \in \mathbb{R}^n | H_x x \le h_x\}, \mathbb{Y}=\{y \in \mathbb{R}^n | H_y y \le h_y\} \subset \mathbb{R}^{n}$, $H_x \in \mathbb{R}^{q_x \times n}, H_y \in \mathbb{R}^{q_y \times n}$. We have $\XX \subseteq \YY$ if and only if
\begin{equation}
\label{eq_lemma}
\exists \Lambda \in \mathbb{R}_+^{q_y \times q_x} \st \Lambda H_x= H_y, \Lambda h_x \le h_y.
\end{equation}
\end{lemma} 
\begin{proof}
The conditions in \eqref{eq_lemma} are equivalent to $\mathbb{X}$ being contained within each closed half-space of the hyperplanes in $\mathbb{Y}$. This condition is verified by checking $q_y$ inequalties: 
\begin{equation}
\label{eq_lemma_max}
\max_{x \in \XX} H_{y,i} x ~~\le~~ h_{y,i}, i=1,\cdots,q_y,
\end{equation}
where $H_{y,i}$ is the $i$'th row of $H_y$ (the same notation applies to $h_y$). By writing the dual of the left hand side in \eqref{eq_lemma_max}, we arrive at
\begin{equation}
\min_{u_i \in \mathbb{R}_+^{q_x}, u_i'H_x=H_{yi}} u_i'h_x ~~\le~~ h_{y,i}, i=1,\cdots,q_y,
\end{equation}
which is equivalent to $\exists u_i \in \mathbb{R}_+^{q_x}, u_i'H_x=H_{yi}$, such that $u_i'h_x \le h_{y,i}$. Let $\Lambda=[u'_1,u'_2,\cdots,u'_{q_y}]$, and \eqref{eq_lemma} immediately follows.  
\end{proof}

Instead of solving $q_y$ linear programs in \eqref{eq_lemma_max}, each with $n$ variables and $\mathcal{O}(q_x)$ constraints, we can solve one linear program in \eqref{eq_lemma} with $\mathcal{O}(q_xq_y)$ variables and constraints. In many cases, the former is more efficient. However, there is merit in \eqref{eq_lemma} as it can be added to a mathematical program to encode $\mathbb{X} \subseteq \mathbb{Y}$. Note that \eqref{eq_lemma} is lossless - it is necessary and sufficient. Now we state the main result of this paper.

\begin{theorem}[Sufficient Conditions for AH-polytope in AH-polytope]
\label{thm_main}
Let $\mathbb{X}=\bar{x} + X \mathbb{P}_x, \mathbb{Y}= \bar{y} + Y \mathbb{P}_y$, where $\mathbb{P}_x=\{x \in \mathbb{R}^{n_x} | H_x x \le h_x\}$ is a full-dimensional polytope, $\mathbb{P}_y=\{y \in \mathbb{R}^{n_y} | H_y y \le h_y\}$, where $q_x, q_y$ are number of rows in $H_x$ and $H_y$, respectively. Then we have $\XX \subseteq \YY$ if:
\begin{equation}
\exists \Gamma \in \mathbb{R}^{n_y \times n_x}, \exists \beta \in \mathbb{R}^{n_y}, \exists \Lambda \in \mathbb{R}_+^{q_y \times q_x}
\end{equation}
such that the following relations hold:
\begin{subequations}
\label{eq_theory_main}
\begin{equation}
\label{eq_param}
X=Y\Gamma, \bar{y}-\bar{x}=Y\beta,  
\end{equation}
\begin{equation}
\Lambda H_x = H_y \Gamma, \Lambda h_x \le h_y + H_y\beta. 
\end{equation}
\end{subequations}
\end{theorem}
\begin{proof}
Since we do not have the hyperplanes of the circumbody, we need to specify the argument similar to \eqref{eq_lemma_max} for all directions in $\mathbb{R}^n$. Therefore, $\mathbb{X} \subseteq \mathbb{Y}$ is equivalent to:
\begin{equation}
\label{eq_support}
\forall c \in \mathbb{R}^n, \max_{x \in \mathbb{P}_x} c'(\bar{x}+Xx) \le \max_{y \in \mathbb{P}_y} c'(\bar{y}+Yy).
\end{equation}
We write the dual of the right hand side to arrive at:
\begin{equation}
\max_{x \in \mathbb{P}_x} c'(\bar{x}+Xx) \le \min_{u \in \mathbb{R}^{q_y}_+, u'H_y=c'Y} u'h_y + c'\bar{y}. 
\end{equation}
Since minimum of the right-hand side set is greater than the maximum of the left-hand side set, it implies that any element of the right-hand side set is greater than any element of the left-hand side set. Therefore, $\forall c \in \mathbb{R}^n, \forall u \in \mathbb{R}^{q_y}_+, u'H_y=c'Y, \forall x \in \mathbb{P}_x$, we must have the following relation:
\begin{equation}
\label{eq_main_condition}
 u'h_y + c'\bar{y} \ge c'Xx + c'\bar{x}. 
\end{equation}

First, we show that the parametrization in \eqref{eq_param} is always possible when $\mathbb{X} \subseteq \mathbb{Y}$ and $\mathbb{P}_x$ is full dimensional. There exists $p^0_x \in \mathbb{P}_x, \epsilon>0, p^0_y \in \mathbb{P}_y$ such that
$$
\bar{x}+Xp^0_x=\bar{y}+Y p^0_y, \bar{x}+X(p^0_x+\epsilon \mathbb{B}_{n_x}) \subset \mathbb{Y},
$$
which implies that $\bar{x}+X(p^0_x+\epsilon e_i)=\bar{y}+Y (p^0_y+\gamma_i),$ for some $\gamma_i \in \mathbb{R}^{n_y}, i=1,\cdots,n_x$, where $e_i$ is the unit vector in the $i$'th Cartesian direction. Therefore, $X e_i=Y\gamma_i, i=1,\cdots,n_x$. Thus, all columns of $X$ lie in $\range(Y)$. Moreover, $\bar{y}-\bar{x}=Xp^0_x-Y\bar{x}$, which also implies that $\bar{y}-\bar{x} \in \range(Y)$.

Now substitute $\bar{y}-\bar{x}$ with $Y \beta$ and $X$ with $Y\Gamma$, and finally $c'Y$ with $u'H_y$ in \eqref{eq_main_condition} to obtain:  
\begin{equation}
\label{eq_nec_suf}
\begin{array}{c}
\forall u \in \mathbb{R}_+^{q_y}, \forall c \in \mathbb{R}^n, u'H_y=c'Y, \forall x\in \mathbb{P}_x, \\
 u'(h_y + H_y\beta - H_y \Gamma x) \ge 0.
 \end{array}
\end{equation}
Until this point, every relation is necessary and sufficient for $\XX \subseteq \YY$. The conditions in \eqref{eq_nec_suf} is bilinear in $u$ and $x$. Furthermore, $H_y\Gamma$ is not necessarily positive definite, so we can not efficiently find the  minimum in \eqref{eq_nec_suf} and check if it is non-negative. Notice that $c$ does not appear directly in the bilinear expression, but we have $u'H=c'Y$. Even though $c$ is allowed to take all values in $\mathbb{R}^n$, $u$ is restricted - we postpone its characterization to the next theorem. 

By dropping $u'H=c'Y$, $u$ becomes only constrained to be non-negative, which means that $H_y + H_y\beta - H_y \Gamma x \ge 0$ for all $x \in \mathbb{P}_x$. Define H-polytope $\mathbb{Q}:= \{z \in \mathbb{R}^{n_y} | H_y \Gamma z \le h_y + H_y\beta \}$. This means that any point in $\mathbb{P}_x$ is also in $\mathbb{Q}$, or $\mathbb{P}_x \subseteq \mathbb{Q}$. Therefore, from Lemma \ref{lemma_HH} we have 
$
\Lambda H_x = H_y \Gamma, \Lambda H_x \le H_y + H_y\beta, \Lambda \ge 0, 
$
and the proof is complete.
\end{proof}

It is worth to note that Theorem \ref{thm_main} has the following geometrical interpretation. It implies that 
\begin{equation}
\label{eq_sides}
- \beta + \Gamma \mathbb{P}_x \subseteq \mathbb{P}_y.
\end{equation}
By left multiplying both sides in \eqref{eq_sides} by $Y$, $\XX \subseteq \YY$ is established.

Before moving to the necessary conditions, we remark that \eqref{eq_nec_suf} is a generalized form of the conditions reported in \cite{kellner2015containment}, where the authors considered containment problems for orthogonal projections of polytopes. {Note that $u$ is restricted to a cone that is the intersection of positive orthant and the linear subspace given by $\{ u \in \mathbb{R}^{q_y} | H_y' u \in \range(Y)\}$. Enumerating all the extreme rays of this cone is not possible in polynomial time}. The main idea in Theorem \ref{thm_main} is avoiding the bilinear terms using parameterization in \eqref{eq_param} and then relaxing $u$ to be in the positive orthant. The conservativeness can also be interpreted as replacing a cone in the positive orthant by the positive orthant itself {(a larger cone that does not depend on $H_y$ and $Y$)}. Now we state a sufficient condition in which \eqref{eq_theory_main} becomes necessary.  

\begin{theorem}
The conditions in \eqref{eq_theory_main} is necessary if 
\begin{equation}
\label{eq_nec}
\range(H_y'^\dagger Y') \oplus \kernel(H_y') = \mathbb{R}^{q_y}. 
\end{equation}
\end{theorem}
\begin{proof}
Since $u'H_y=c'Y$ and $c$ taking all possible values in $\mathbb{R}^n$, we have $u \in \range(H_y'^\dagger Y') \oplus {\ker(H'_y)} \cap \mathbb{R}_+^{q_y}$. If the subspace $\range(H_y'^\dagger Y') \oplus {\ker(H_y')}$ is actually the whole space, then $u$ only becomes constrained by $u \in \mathbb{R}_+^{q_y}$. Thus, $u'H_y=c'Y$ can be safely dropped without any loss.  
\end{proof}

We now provide a sufficient condition for \eqref{eq_nec} to hold.
 
\begin{corollary}
\label{cor_inverse}
The conditions in \eqref{eq_theory_main} {are} necessary and sufficient if $Y$ has a left-inverse, i.e. it has linearly independent columns.  
\end{corollary}
\begin{proof}
If $Y$ has linearly independent columns, then $\ker(Y)=\{0\}$, or $Y'$ has linearly independent rows. Therefore, $\range(H_y'^\dagger Y')=\range( H_y'^\dagger)$. Then by rank-nullity theorem of linear algebra, we have $\range( H_y'^\dagger) \oplus ker( H_y')=\mathbb{R}^{q_y}$. 
\end{proof}

\begin{corollary}
\label{cor_AH}
The AH-polytope in H-polytope containment problem can be decided in polynomial time.  
\end{corollary}
\begin{proof}
This is a special case of Corollary \ref{cor_inverse} with $Y=I$, which reduces to the linear program in \eqref{eq_theory_main}, which is lossless. 
\end{proof}
Corollary \ref{cor_inverse} is not surprising as it was already mentioned that if $Y$ has a left inverse, one can replace $Y\mathbb{P}_y$ by H-polytope $\mathbb{Q}=\{z \in \mathbb{R}^n | H_y Y^{\dagger} z \le h_y\}$. 
Corollary \ref{cor_AH} is a known result in the literature. A version was derived in \cite{kellner2015containment}. It can also be proved using other techniques in basic convex analysis \cite{rockafellar2015convex}. The authors in \cite{rakovic2007optimized} also derived linear encodings for a specific version of AH-polytope in H-polytope containment. 

\begin{remark}
{It is worth to note that while $\Lambda, \Gamma, \beta$ are the decision variables in \eqref{eq_theory_main}, $X$ and $\bar{x}$ do appear linearly with respect to these variables, while $H_x,  H_y$ and $Y$ appear in bilinear terms with  $\Lambda, \Gamma, \beta$. This observation has practical virtues. It implies that one can consider $\bar{x}, X$ as decision variables and still cast the containment problem as linear constraints. This is particularly useful when $\mathbb{X} \subseteq \mathbb{Y}$ is one of the constraints of a (possibly mixed-integer) convex program. For instance, we exploit this fact in Section. \ref{sec_control} to synthesize trajectories of AH-polytopes (parameterized with multiple instances $\bar{x},{X}$) that end within a given polytope $\mathbb{Y}$ subject to mixed-logical dynamical \cite{Bemporad1999} constraints.   
}     
\end{remark}


\section{Special Cases: Zonotopes, Minkowski Sums, Convex Hulls, and Disjunctive Containment}
\label{sec_special}
In this section, we elaborate on the usefulness of Theorem \ref{thm_main} for a number of practically relevant special cases. For the ease of readability, we adopt the notation convention used throughout Section \ref{sec_main}, unless stated otherwise, in the rest of the paper.

\subsection{Zonotopes}
\label{sec_zono}
\begin{theorem}[Zonotope in Zonotope Containment]
\label{thm_zono}
We have$\langle \bar{x},X\rangle \subseteq \langle \bar{y},Y \rangle$, $X \in \mathbb{R}^{n \times n_x}, Y \in \mathbb{R}^{n \times n_y}$, if there exists $\Gamma \in \mathbb{R}^{n_y \times n_x}, \beta \in \mathbb{R}^{n_y}$ such that:
\begin{equation}
\label{eq_zono}
X = Y \Gamma, \bar{y}-\bar{x} = Y \beta, \left\| (\Gamma,\beta) \right\|_\infty \le 1. 
\end{equation}
\end{theorem}
\begin{proof}
Using \eqref{eq_theory_main}, we have (i) ${\Lambda [I,-I]}= [I,-I] \Gamma, \Lambda \ge 0, \Lambda \underbar{1} \le  \underbar{1} + [I,-I] \beta$, and (ii) $X=Y \Gamma, \bar{y}-\bar{x} = Y \beta$. We need to show that the existence of $\Lambda \in \mathbb{R}_+^{2n_y \times 2n_x}$ such that (i) holds is equivalent to $\left\| (\Gamma,\beta) \right\|_\infty \le 1$.  

Let $\Lambda=[(\Lambda_1,\Lambda_2),(\Lambda_3,\Lambda_4)], \Lambda_i \in \mathbb{R}_+^{n_y \times n_x}, i=1,2,3,4$. Then we have: 
\begin{subequations}
\begin{equation}
\label{eq_lambda_zono_1}
\Lambda_1-\Lambda_2=\Lambda_4-\Lambda_3=\Gamma,
\end{equation}
\begin{equation}
\label{eq_lambda_zono_2}
(\Lambda_1+\Lambda_2)\underbar{1}-\beta \le 1, (\Lambda_3+\Lambda_4)\underbar{1}+\beta \le 1.
\end{equation}
\end{subequations}
Since all entries in $\Lambda$ are non-negative, from \eqref{eq_lambda_zono_1} we have $|\Gamma|\le \Lambda_1+\Lambda_2$ and $|\Gamma|\le \Lambda_3+\Lambda_4$, which by replacing in \eqref{eq_lambda_zono_2} we obtain $|\Gamma|\underbar{1}-\beta \le \underbar{1}$ and $|\Gamma|\underbar{1}+\beta \le \underbar{1}$. Thus, $|\Gamma| \underbar{1}+|\beta| \le \underbar{1} \Leftrightarrow \left\| (\Gamma,\beta) \right\|_\infty \le 1$, and the proof is complete.  

\end{proof}

\begin{example}
\label{example_zonotopes}
Consider two zonotopes in $\mathbb{R}^2$:
\begin{equation*}
\bar{x}=\left(
\begin{array}{c}
0 \\ 1
\end{array}
\right),
X=\left(
\begin{array}{ccccc}
1& 0 & 0 & 1& 1 \\
0& -1 & 0 & -1& -3 \\
\end{array}
\right),
\end{equation*}
\begin{equation*}
\bar{y}=\left(
\begin{array}{c}
1 \\ 0
\end{array}
\right),
Y=\left(
\begin{array}{cccccc}
1& 0 & 1 & 1& 1 & 2\\
0& 1 & 1 & -1& 3 & -2 \\
\end{array}
\right).
\end{equation*}
The containment of $\mathbb{Z}_x \subseteq \mathbb{Z}_y$ is verified by Theorem \eqref{thm_zono} and is also illustrated in Fig. \ref{fig_zonotopes} [Left]. 
If we drop the last column from $Y$, containment no longer holds, which follows from the infeasibility of \eqref{eq_zono} and is also shown in Fig. \ref{fig_zonotopes} [Right].
\begin{figure}
\centering
\includegraphics[width=0.22\textwidth]{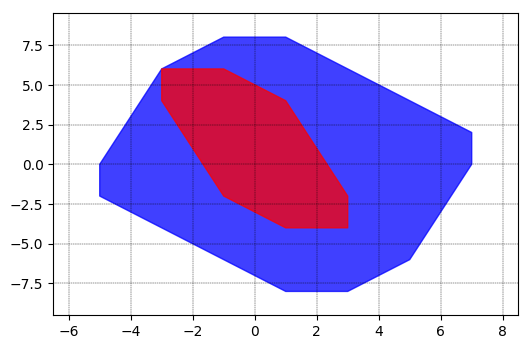}
\includegraphics[width=0.22\textwidth]{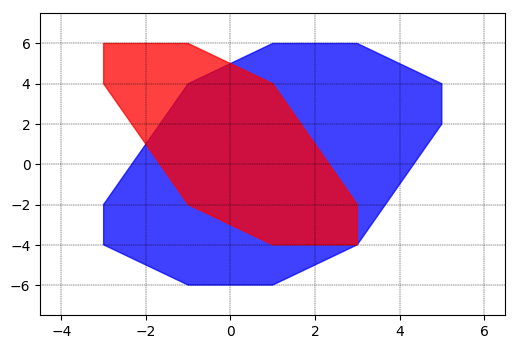}
\label{fig_zonotopes}
\caption{Example \ref{example_zonotopes}: Zonotope Containment Problem: [left] $\mathbb{Z}_x \subseteq \mathbb{Z}_y$, [Right] $\mathbb{Z}_x \not \subseteq \mathbb{Z}_y^*$, where the last column of $Y$ is dropped.}
\end{figure}
\end{example}

{\bf On the Conservativeness of Theorem \ref{thm_zono}:}
Perhaps surprisingly, we found that Theorem \ref{thm_zono} is often lossless. We found manually searching for a counterexample, where containment holds but Theorem \ref{thm_zono} fails to verify it, to be non-trivial. It remains an open problem whether a lossless condition for zonotope containment without relying on vertex/hyperplane enumeration exists. 

In order to characterize conservativeness, given zonotopes $\mathbb{Z}_x$ and $\mathbb{Z}_y$, we enumerate the vertices of $\mathbb{Z}_x$ - there are at most $2^{n_x}$ of them, where {$n_x$} is the number of columns in $X$. Then we solve two linear programs: I) the {maximum} $\lambda_{\text{lossless}}$ such that all the vertices of {$\lambda_{\text{lossless}} \mathbb{Z}_x$} are inside $\mathbb{Z}_y$, and II) the {maximum} $\lambda_{\text{Theorem}~ \ref{thm_zono}}$ such that \eqref{eq_zono} holds for {$\lambda_{\text{Theorem}~ \ref{thm_zono}} \mathbb{Z}_x \subseteq \mathbb{Z}_y$}. We introduce the loss function 
\begin{equation}
\text{loss}=(\lambda_{\text{lossless}}- \lambda_{\text{Theorem}~ \ref{thm_zono}})/\lambda_{\text{lossless}}.
\end{equation}
If $\text{loss}=0$ for a certain zonotope containment problem, then Theorem \ref{thm_zono} does not introduce any conservativeness.

We randomly generated over 10000 zonotopes in $\mathbb{R}^n, n=3,\cdots,10$, with random number of generator columns (uniformly sampled between $n$ and $12$) for the inbody and the circumbody, and random values of generator matrix entries uniformly and independently chosen between $-1$ and $1$, and performed a statistical analysis. The loss was smaller than $0.01$  for nearly $98\%$ of the zonotopes. The histogram of loss values for all zonotopes is shown in Fig. \ref{fig_loss} [Left], and for different dimensions in Fig. \ref{fig_loss} [Right]. The histogram is so skewed toward large numbers of zonotopes with small loss values that we used logarithmic scale for meaningful illustration. We observed a trend, albeit not very strong, of loss values getting larger with zonotope dimension $n$. We never observed a loss greater than $0.1$. We intuitively expect that at very high dimensions $(n > 10)$, the loss may be significant, but verifying this fact requires zonotope vertex/hyperplane enumeration, which is not possible for very large values of $n$. 

\begin{figure}
\centering
\includegraphics[width=0.22\textwidth]{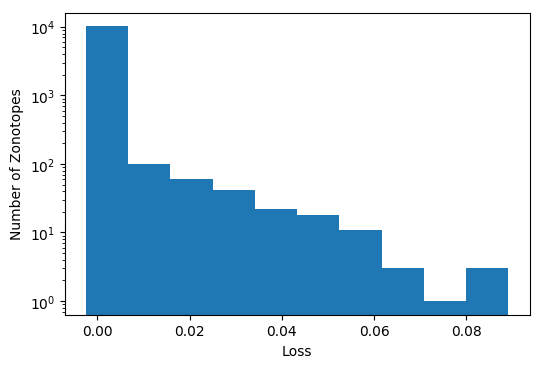}
\includegraphics[width=0.22\textwidth]{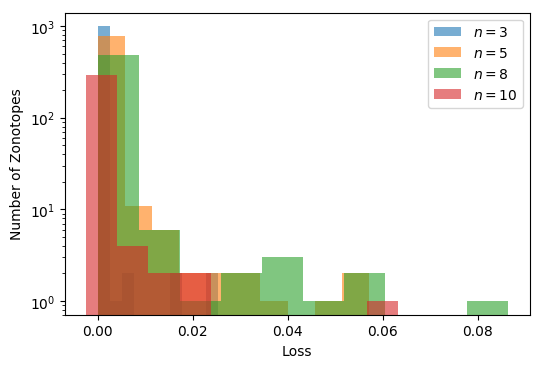}
\caption{Histograms of Loss values for Theorem \ref{thm_zono}}
\label{fig_loss}
\end{figure}

With a randomized search over rational generators, we found the following counterexample in $n=3$. We never found a counterexample with $n=2$. 

\begin{example}
Consider two zonotopes in $\mathbb{R}^3$ with centroids at origin:
\begin{equation*}
X=\left(
\begin{array}{ccc}
5& -1 & 2 \\
-4 & -2 & 2 \\
4 & -1 & -4 \\
\end{array}
\right),
\end{equation*}
\begin{equation*}
Y=\left(
\begin{array}{ccccc}
4& 0 & -4 & 1& 0 \\
-3& 0 & 0 & 4& 1 \\
1& -4 & -5 & -1& -3 \\
\end{array}
\right).
\end{equation*}
One can verify $\mathbb{Z}_x \subseteq \mathbb{Z}_y$ through checking all the vertices of $\mathbb{Z}_x$. However, Theorem \ref{thm_zono} fails to establish containment as \eqref{eq_zono} is infeasible. However, feasibility is gained by scaling $\mathbb{Z}_x$ by 0.9915. The loss for this case is $0.0085$.
\end{example}

\begin{remark}
The authors in \cite{han2016enlarging} provided a sufficient condition based on linear matrix inequalities (LMI) for the zonotope containment problem in which the zonotopes were in their infinity-norm representation as $\mathbb{Z}=\{ x \in \mathbb{R}^n | \|H(x-\bar{x})\|_\infty\}$. However, in this case, the H-polytope form of zonotopes is already available as $\mathbb{Z}= \{ x \in \mathbb{R}^n | [I -I] H(x-\bar{x}) \le \underbar{1} \}.$ Therefore, Lemma \ref{lemma_HH} provides necessary and sufficient conditions for this restricted classes of zonotopes and a sufficient condition based on an LMI characterization is not required.
\end{remark}


\subsection{Minkowski Sums}
Minkowski sums of polytopes are widely used in computational geometry, collision detection, and robust control. For example, the set of reachable states of an uncertain dynamical system can be characterized as the Minkowski sum of set-valued uncertain effects accumulated over time. The Minkowski sum of multiple H-polytopes can be rewritten as an AH-polytope. For example, let $\mathbb{P}_1=\{x \in \mathbb{R}^n | H_1 x \le h_1\} ,\mathbb{P}_2=\{x \in \mathbb{R}^n | H_2 x \le h_2\}$, be two H-polytopes. Then $\mathbb{P}_1 \oplus \mathbb{P}_2$ is equivalent to:
\begin{equation}
\mathbb{P}_1 \oplus \mathbb{P}_2= (I,I) \mathbb{P}_{sum},
\end{equation} 
where $\mathbb{P}_{sum}=\{ z \in \mathbb{R}^{2n} | H_{sum} z \le h_{sum} \}$, and
$$
H_{sum}=
\left(
\begin{array}{cc}
H_1 & 0 \\
0 & H_2 \\
\end{array}
\right), h_{sum}= 
\left(
\begin{array}{c}
h_1 \\
h_2 \\
\end{array}
\right).
$$
If the Minkowski sum is the inbody, we know from Corollary \ref{cor_AH} that there exists a lossless encoding. The following result was also reported in \cite{rakovic2007optimized}. 

\begin{proposition}[Minkowski sum as the inbody]
\label{prop_minko_in}
We have $\bar{x}+\bigoplus_{i=1}^N X_i \mathbb{P}_{x,i} \subseteq \mathbb{P}_y$, if and only if the $\exists \Lambda_i \in \mathbb{R}_+^{q_{y} \times q_{x,i}}, i=1,\cdots,N$, such that the following conditions hold:
\begin{equation}
\begin{array}{c}
\Lambda_i H_{x,i} = H_y X_i, i=1,\cdots,N, \displaystyle \sum_{i=1}^N \Lambda_i h_{x,i} \le h_y - H_y \bar{x}.
\end{array}
\end{equation}
\end{proposition}

The case of Minkowski sum being the circumbody is more complicated. We want to encode $ \mathbb{P}_x \subseteq \mathbb{P}_1 \oplus \mathbb{P}_2$. Using Theorem \ref{thm_main}, we arrive at the following conditions:
$$
\begin{array}{c}
I=(I,I)\Gamma, 0=(I,I)\beta, \\
\Lambda H_x = H_{sum} \Gamma, \Lambda h_x \le h_{sum} + H_{sum} \beta.
\end{array} 
$$
By writing $\Gamma=[\Gamma_1,\Gamma_2], \beta=[\beta_1,\beta_2], \Lambda=[\Lambda_1,\Lambda_2]$, we have the following relations:
\begin{equation*}
\begin{array}{c}
 \Lambda_i H_x=H_i \Gamma_i, \Lambda_i h_x \le h_i + H_i \beta_i, i=1,2.\\ 
\end{array}
\end{equation*}
This solution has a geometrical interpretation. It implies there exists $\mathbb{X}_i \subseteq \mathbb{P}_i, i=1,2$, where $\mathbb{X}_i = - \beta_i + \Gamma_i \mathbb{P}_x$. In other words, two affine transformations of $\mathbb{P}_x$ exist, one in $\mathbb{P}_1$ and the other in $\mathbb{P}_2$, such that the sum of the transformation matrices equals $I$ and their offsets add up to zero. This idea can be generalized to $N$ polytopes. 

\begin{proposition}[Minkowski Sum as the Circumbody]
\label{result_minkowski_right}
Let $N+1$ polytopes $\mathbb{P}_x, \mathbb{P}_{y,i} \subset \mathbb{R}^{n},i=1,\cdots, N,$. Then 
$$\bar{x}+ X \mathbb{P}_x \subseteq \bigoplus_{i=1}^N (\bar{y}_i + Y_i \mathbb{P}_{y,i}) $$
if there exists $\Lambda_i, \beta_i, \Gamma_i, i=1,\cdots,N$, with appropriate sizes from \eqref{eq_theory_main}, such that 
\begin{equation}
\begin{array}{c}
\Lambda_i H_x= H_{y,i} \Gamma_i, \Lambda_i h_x \le h_{y,i} + H_{y,i} \beta_i, i=1,\cdots,N, \\
\displaystyle \sum_{i=1}^N \bar{y_i}-Y_i \beta_i = \bar{x}, 
\displaystyle \sum_{i=1}^N Y_i \Gamma_i = X. 
 \end{array}
\end{equation}
\end{proposition}
\begin{proof}
The proof follows from writing $\bigoplus_{i=1}^N Y_i \mathbb{P}_{y,i}$ as $(Y_1,\cdots,Y_N) P_{sum}$, where $H_{sum}=\blk(H_{y,1},\cdots,H_{y,N})$ and $h_{sum}=[h_{y,1},\cdots,h_{y,N}]$. The rest of the proof follows from Theorem \ref{thm_main}.    
\end{proof}

The condition in Proposition \ref{result_minkowski_right} is conservative. Otherwise, we would have been able to verify $\mathbb{P}_x=\mathbb{P}_1 \oplus \mathbb{P}_2$ for H-polytopes in polynomial time, a problem that is proven to be NP-complete \cite{tiwary2008hardness}. However, the result may still be useful depending on the application. For example, the trivial proposition $\mathbb{P}_i \subseteq \mathbb{P}_1 \oplus \mathbb{P}_2, i=1,2$ always follows from Proposition \ref{result_minkowski_right} when $\mathbb{P}_1$ and $\mathbb{P}_2$ contain the origin. The next example demonstrates the limitations of Proposition \ref{result_minkowski_right}.

\begin{example}
\label{example_minko}
Consider two H-polytopes in $\mathbb{R}^2$, $\mathbb{P}_1=\{[x_1,x_2] | x_1+x_2 \le 1, -x_1+x_2 \le 1, x_2 \ge 0]\}$, a triangle, and $\mathbb{P}_2=\{[x_1,x_2] | |x_1| \le 0.1, -1\le x_2 \le 0]\}$, a box stretching from $[0,-1]$ to $[0,0]$ with width 0.1 - see Fig. \ref{fig_minkowski} [Top Left]. Let $\mathbb{P}_{sum}=\mathbb{P}_1 \oplus \mathbb{P}_2$. For this 2D example, it is easy to enumerate the hyperplanes of $\mathbb{P}_{sum}$. Proposition \ref{result_minkowski_right} fails to establish $\mathbb{P}_{sum} \subseteq \mathbb{P}_1 \oplus \mathbb{P}_2$, but it verifies $\lambda \mathbb{P}_{sum} \subseteq \mathbb{P}_1 \oplus \mathbb{P}_2$ for $\lambda \le 0.68$, which manifests the loss of Proposition \ref{result_minkowski_right}. The decomposed H-polytopes, as described in Proposition \ref{result_minkowski_right}, are illustrated in Fig. \ref{fig_minkowski}.   

\begin{figure}
\centering
\includegraphics[width=0.22\textwidth]{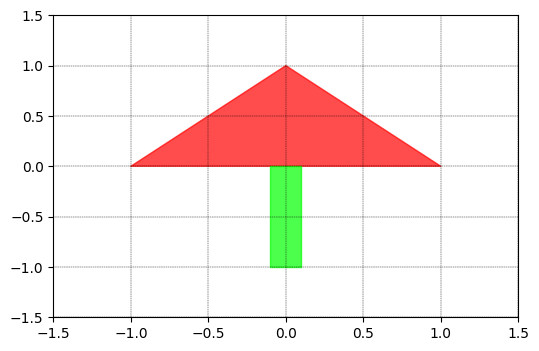}
\includegraphics[width=0.22\textwidth]{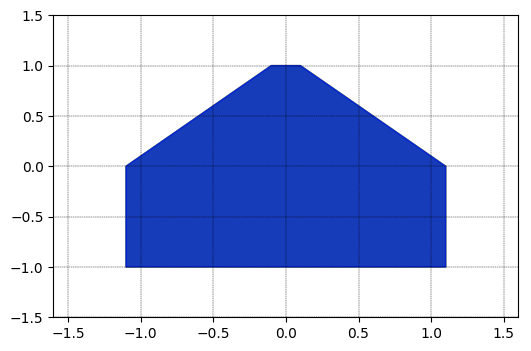}
\includegraphics[width=0.22\textwidth]{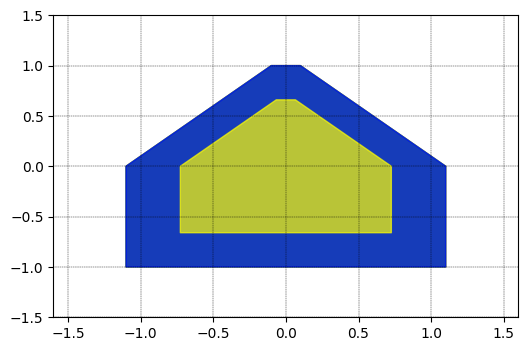}
\includegraphics[width=0.22\textwidth]{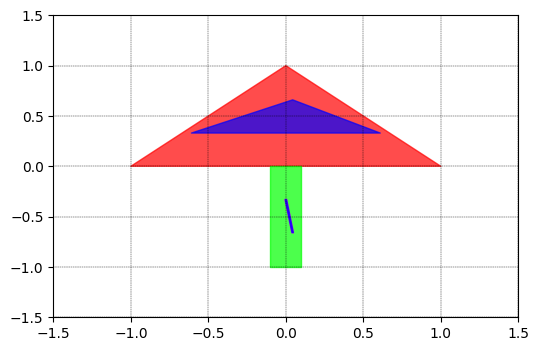}
\label{fig_minkowski}
\caption{Example \ref{example_minko}: [Top Left] Two H-polytopes $\mathbb{P}_1,\mathbb{P}_2$ [Top Right] $\mathbb{P}_1 \oplus \mathbb{P}_2$ [Bottom Left] $0.66(\mathbb{P}_1 \oplus \mathbb{P}_2$) [Bottom Right] Decomposition of Polytopes $\Gamma_1\mathbb{P}_1+\bar{x}_1 \subseteq \mathbb{P}_1$ and $\Gamma_2\mathbb{P}_1+\bar{x}_2 \subseteq \mathbb{P}_2$ (both shown in blue) such that $\Gamma_1+\Gamma_2=I$ and $\bar{x}_1+\bar{x}_2=0$.}
\end{figure}

\end{example}


\subsection{Convex hulls}
Enforcing containment for convex hulls is useful when the desired property is closely related to convexity. 
  When the convex hull of multiple polytopes is the inbody, the containment problem becomes straightforward as it basically implies that each polytope is an inbody itself - we obtain a series of polytope containment problems. However, when the convex hull is the circumbody, there is a significant departure in the complexity. Note that H-polytope in V-polytope containment, a co-NP-complete problem, is a special case of the convex hull being the circumbody, where each point can be represented as a singular H-polytope. 

\begin{proposition}[Convex Hull of H-polytopes as the Circumbody]
\label{prop_convexhull}
Given $N+1$ polytopes $\mathbb{P}_x,\mathbb{P}_{y,i} ,i=1,\cdots, N$, we have 
$$\bar{x} + X \mathbb{P}_x \subseteq \convh(\{\mathbb{P}_{y,i}\}_{i=1,\cdots,N})$$
 if there exists $\Lambda_i \ge 0, \lambda_i \ge 0, \beta_i, \Gamma_i, i=1,\cdots,N$, such that
\begin{equation}
\label{eq_convexhull_H}
\begin{array}{c}
\Lambda_i H_x= H_{y,i} \Gamma_i, \Lambda_i h_x \le \lambda_i h_{y,i} + H_{y,i} \beta_i, i=1,\cdots,N,  \\
\displaystyle \sum_{i=1}^N \lambda_i=1,
\displaystyle \sum_{i=1}^N -\beta_i=\bar{x},
\displaystyle \sum_{i=1}^N \Gamma_i=X.
 \end{array}
\end{equation}
\end{proposition} 
\begin{proof}
The convex hull is an AH-polytope as follows:
 $$\convh(\{\mathbb{P}_{y,i}\}_{i=1,\cdots,N})= (I,\cdots,I,0,\cdots,0) P_{hull},$$ where $H_{hull}=[(\blk(H_{y,1},\cdots,H_{y,N}),-\blk(h_{y,1},\cdots,h_{y,N}))$, $(0,-I),(0,\underbar{1}'),(0,-\underbar{1}')],$ and $h_{hull}=[0,\cdots,0,1,1]$. Let $\Gamma=[\Gamma_1,\cdots,\Gamma_N,\gamma_1,\cdots,\gamma_N],\beta=[\beta_1,\cdots,\beta_N,-\lambda_1,\cdots,-\lambda_N]$, and $\Lambda=[\Lambda_1,\cdots,\Lambda_N, \eta_1,\cdots,\eta_N, \zeta_1,\zeta_2]$. Then the following relations satisfy Theorem \ref{thm_main}:
$$
\begin{array}{c}
X= \displaystyle \sum_{i=1}^N \Gamma_i, -\bar{x}=\displaystyle \sum_{i=1}^N \beta_i, \\
\Lambda_i H_x = H_{y,i} \Gamma_i - h_{y,i} \gamma_i, \eta_i H_x = - \gamma_i,\\
\Lambda_i h_x \le \lambda_i h_{y,i} + H_{y,i} \beta_i, \eta_i h_x \le \lambda_i, i=1,\cdots,N, \\
\zeta_1 H_x = \displaystyle \sum_{i=1}^N \gamma_i, \zeta_2 H_x=-\sum_{i=1}^N \gamma_i, \\
 \zeta_1 h_x \le 1 -\displaystyle \sum_{i=1}^N \lambda_i, \zeta_2 h_x \le -1 + \displaystyle \sum_{i=1}^N \lambda_i.
\end{array}
$$
A possible solution is $\gamma_i,\eta_i=0, i=1,\cdots,N, \zeta_1=\zeta_2=0$. Thus, the relations above yield $\lambda_i \ge 0, i=1,\cdots,N$, and  $\sum_{i=1}^N \lambda_i=1$, and the proof is complete. 
\end{proof}
Note that Proposition \ref{result_minkowski_right} is a special case of Proposition \ref{prop_convexhull} with $\lambda_i=1, i=\cdots,N$, while dropping the constraint $\sum_{i=1}^N \lambda_i=1$. Proposition \ref{prop_convexhull} also has the following geometrical interpretation:
\begin{equation}
\label{eq_interpret}
-\beta_i + \Gamma_i \mathbb{P}_x  \subseteq \lambda_i \mathbb{P}_{y,i}.
\end{equation}
We use \eqref{eq_interpret} to generalize Proposition \ref{prop_convexhull} to the convex hull of AH-polytopes. The reason that we present Proposition \ref{prop_convexhull} independently, and not as a special case of the following result, is explained in Sec. \ref{sec_disjunctive}. First, we state the following simple yet useful lemma. 

\begin{lemma}
\label{lemma_trivial}
Given $\mathbb{S} \subset \mathbb{R}^n$ and matrices $A_1,\cdots,A_N$ with appropriate dimensions, the following relation holds:
\begin{equation}
(\sum_{i=1}^N A_i) \mathbb{S} \subseteq \bigoplus_{i=1}^N A_i \mathbb{S}.
\end{equation}
\end{lemma}
\begin{proof}
Straightforwardly follows from the definitions.  
\end{proof}

Note that by taking the Minkowski sums of both sides in \eqref{eq_interpret} over $i=1,\cdots,N$, and using Lemma \ref{lemma_trivial}, \eqref{eq_convexhull_H} follows.

\begin{corollary}[Convex Hull of AH-Polytopes as the Circumbody]
\label{prop_convexhull_AH}
Given $N+1$ polytopes $\mathbb{P}_x,\mathbb{P}_{y,i} ,i=1,\cdots, N$, then 
$$\bar{x} + X \mathbb{P}_x \subseteq \convh(\{\bar{y}_i+ Y_i \mathbb{P}_{y,i}\}_{i=1,\cdots,N})$$
 if there exists $\Lambda_i \ge 0, \lambda_i \ge 0, \beta_i,\Gamma_i, i=1,\cdots,N$, such that
\begin{equation}
\label{eq_convexhull_AH}
\begin{array}{c}
\Lambda_i H_x= H_{y,i} \Gamma_i, \Lambda_i h_x \le \lambda_i h_{y,i} + H_{y,i} \beta_i,\\
i=1,\cdots,N,  (1,\bar{x},X)=\displaystyle \sum_{i=1}^N (\lambda_i,\lambda_i \bar{y}_i-Y_i \beta_i,Y_i \Gamma_i). 
 \end{array}
\end{equation}
\end{corollary} 
\begin{proof}
Consider $N$ polytopes $\bar{x}_i + X_i \mathbb{P}_x, i=1,\cdots,N$, such that:
\begin{equation}
\label{eq_individual}
\bar{x}_i + X_i \mathbb{P}_x \subseteq \lambda_i (\bar{y}_i + Y_i \mathbb{P}_{y,i}).
\end{equation}
By taking the Minkowski sum of both sides, we have, by the virtue of Lemma \ref{lemma_trivial}, $\sum_{i=1}^N \bar{x}_i + (\sum_{i=1}^N X_i) \mathbb{P}_x \subseteq \convh(\{\bar{y}_i+ Y_i \mathbb{P}_{y,i}\}_{i=1,\cdots,N})$. It suffices to have $(\bar{x},X)=\sum_{i=1}^N (\bar{x}_i, X_i)$ and apply Theorem \ref{thm_main} to \eqref{eq_individual}. Replace $\lambda \mathbb{P}_{y,i}$ by $\{y \in \mathbb{R}^{n_i} | H_{y,i} y \le \lambda_i h_{y,i}\}$. We have $\lambda_i \bar{y}_i-\bar{x}_i=Y_i \beta_i$ and $X_i=Y_i \Gamma_i, \Lambda_i H_x = H_{y,i} \Gamma_i$, and $\Lambda_i h_x \le \lambda_i h_{y,i} + H_{y,i} \beta_i$. With some algebraic manipulation for substituting $X_i$ and $\bar{x}_i$, \eqref{eq_convexhull_AH} follows.  
\end{proof}

While Corollary \ref{prop_convexhull} provides a sufficient condition, it may be very conservative. In the extreme case where $\mathbb{P}_{y,i}=\{0\}, i=1,\cdots,N$, the problem reduces to AH-polytope in V-polytope containment problem. It follows from \eqref{eq_interpret} that \eqref{eq_convexhull_AH} returns infeasibility unless the inbody is a point and is in the convex hull of the circumbody points $\bar{y}_i, i=1,\cdots,N$ - the loss is $100 \%$ as no volumetric inbody can be considered.

\subsection{Disjunctive Polytope Containment Problem}
\label{sec_disjunctive}
Proposition \ref{result_minkowski_right} and Proposition \ref{prop_convexhull} are special cases of Theorem \ref{thm_main}. However, Proposition \ref{prop_convexhull}, in particular, leads to a useful necessary and sufficient conditions for a class of disjunctive containment problems, which is stated in this section. This problem has strong relevance to encoding logical phenomena such as those in hybrid systems. 

\begin{proposition}[Disjunctive Containment Problem]
\label{prop_disjuntcive}
Given $N+1$ polytopes $\mathbb{P}_x,\mathbb{P}_{y,i} ,i=1,\cdots, N$, then we have
\begin{equation}
\label{eq_disjunctive}
\bigvee_{i=1}^N ( \bar{x} + X \mathbb{P}_x \subseteq \mathbb{P}_{y,i})
\end{equation}
 if and only if there exists $\Lambda_i \ge 0, \delta_i \in \{0,1\}, \beta_i,\Gamma_i, i=1,\cdots,N$, such that
\begin{equation}
\begin{array}{c}
\Lambda_i H_x= H_{y,i} \Gamma_i, \Lambda_i h_x \le \delta_i h_{y,i} - H_{y,i} \beta_i, i=1,\cdots,N,  \\
\displaystyle \sum_{i=1}^N  \delta_i =1,
\displaystyle \sum_{i=1}^N \beta_i=\bar{x},
\displaystyle \sum_{i=1}^N \Gamma_i=X.
 \end{array}
\end{equation}
\end{proposition}
\begin{proof}
(Sufficiency): let $\delta_i=1$ and $\delta_j=0, \forall j \in \{1,\cdots,N\} \setminus \{i\}$. Then we have $\beta_j + \Gamma_j \mathbb{P}_x \subseteq \delta_j \mathbb{P}_{y,j}$, which enforces $\Gamma_j=0, \Lambda_j=0$, and $\beta_j=0$ for $\forall j \in \{1,\cdots,N\} \setminus \{i\}$, and establishes $\beta_i + \Gamma_i \mathbb{P}_x \subseteq \mathbb{P}_{y,i}$ and $\beta_i=\bar{x}$.
(Necessity): if there exists $i\in \{1,\cdots,N\}$ such that $\bar{x}_i + X \mathbb{P}_x \subseteq \mathbb{P}_{y,i}$, then there exists $\Lambda_i$ such that $\Lambda_i H_x = H_{y,i} X, \Lambda_i h_x \le h_{y,i} - H_{y,i} \bar{x}$. Then $\Lambda_j=0, \beta_j=0, \Gamma_j=0, \forall j \in \{1,\cdots,N\} \setminus \{i\}$ must be a feasible solution for \eqref{eq_disjunctive}.  
\end{proof}
 
An alternative way to encode \eqref{eq_disjunctive} is using Lemma \ref{lemma_HH} and big-M method. The main advantage of Proposition \ref{prop_disjuntcive} is that its binary relaxation yields the tightest feasible set, which is $\convh(\{\mathbb{P}_{y,i}\}_{i=1,\cdots,N})$. However, more variables and constraints are introduced by Proposition \ref{prop_disjuntcive} than the big-M method. The performance of two formulations is highly problem dependent, and there is no clear winner. The interested reader is referred to \cite{marcucci2018mixed} for a detailed discussion of the various mixed-integer formulations for hybrid systems. 
The following consequences of Corollary \ref{prop_convexhull} and Lemma \ref{prop_disjuntcive} are tailored for zonotopes.

\begin{corollary}[Zonotope in the Convexhull of Zonotopes]
\label{cor_zono_hull}
Let $\langle \bar{y}_i,Y_i \rangle, i=1,\cdots,N$, be N zonotopes. Then we have
\begin{equation}
\langle \bar{x},X \rangle \subseteq \convh(\bigcup_{i=1,\cdots,N}\{ \langle \bar{y}_i,Y_i \rangle \})  
\end{equation}
if there exists $\Gamma_i, \beta_i, \lambda_i \ge 0, i=1,\cdots,N$, such that
\begin{equation}
\begin{array}{c}
\|(\Gamma_i|\beta_i)\|_{\infty} \le \lambda_i, i=1,\cdots,N,\\
\displaystyle \sum_{i=1}^N \lambda_i=1,
\displaystyle \sum_{i=1}^N Y_i \beta_i + \lambda_i \bar{y}_i= \bar{x},
\displaystyle \sum_{i=1}^N Y_i \Gamma_i=X.
\end{array}
\end{equation}
\end{corollary}

\begin{corollary}[Disjunctive Containment of Zonotopes]
\label{cor_zono_dis}
Let $\langle \bar{y}_i,Y_i \rangle, i=1,\cdots,N$, be N zonotopes. Then we have
\begin{equation}
\bigvee_{i=1}^N (\langle \bar{x},X \rangle \subseteq \langle \bar{y}_i,Y_i \rangle \})
\end{equation}
if there exists $\Gamma_i, \beta_i, \delta_i \in \{0,1\}, i=1,\cdots,N$, such that
\begin{equation}
\begin{array}{c}
\|(\Gamma_i|\beta_i)\|_{\infty} \le \delta_i, i=1,\cdots,N, \\
\displaystyle \sum_{i=1}^N \delta_i=1,
\displaystyle \sum_{i=1}^N Y_i \beta_i + \delta_i \bar{y}_i= \bar{x},
\displaystyle \sum_{i=1}^N Y_i \Gamma_i=X.
\end{array}
\end{equation}
\end{corollary}


\begin{table*}
\centering
\caption{Summary of Relations. The shaded rows are contributions of this paper.}
\resizebox{\textwidth}{!}{
\begin{tabular}{|c|c|c|c|}
\hline
{\bf Relation} & {\bf Linear Encoding} (all $\Lambda$ matrices are non-negative) & {\bf Reference} & {\bf Lossless} 
\\
\hline
$\mathbb{P}_x \subseteq \mathbb{P}_y$ & $\Lambda H_x = H_y,~~\Lambda h_x \le h_y$ & Lemma \ref{lemma_HH}  & \checkmark
\\

\hline
$\bar{x} + X \mathbb{P}_x \subseteq \mathbb{P}_y$ & $\Lambda H_x = H_y X,~~\Lambda h_x \le h_y - H_y \bar{x}$ & Corollary \ref{cor_AH} & \checkmark
\\
\hline
\rowcolor{black!10}
$\bar{x} + X \mathbb{P}_x \subseteq \bar{y} + Y \mathbb{P}_y$ & $X=Y\Gamma, \bar{y}-\bar{x}=Y\beta, \Lambda H_x = H_y \Gamma,~~\Lambda h_x \le h_y + H_y \beta$ & Theorem \ref{thm_main} & \text{\sffamily X} 
\\
\hline
\rowcolor{black!10}
$\langle \bar{x},X\rangle \subseteq \langle \bar{y},Y \rangle$ 
& 
$X = Y \Gamma, \bar{y}-\bar{x} = Y \beta, \left\| (\Gamma,\beta) \right\|_\infty \le 1.$ & Theorem \ref{thm_zono} & \text{\sffamily X} \\

\hline
$\bar{x}+\displaystyle \bigoplus_{i=1}^N X_i \mathbb{P}_{x,i} \subseteq \mathbb{P}_y$ & $\Lambda_i H_{x,i} = H_y \Gamma_i, i=1,\cdots,N, \displaystyle \sum_{i=1}^N \Lambda_i h_{x,i} \le h_y - H_y \bar{x}.$ & Proposition \ref{prop_minko_in} & \checkmark 
\\

\hline
\rowcolor{black!10}

$\bar{x}+ X \mathbb{P}_x \subseteq \displaystyle \bigoplus_{i=1}^N (\bar{y}_i + Y_i \mathbb{P}_{y,i}) $ & 
$\begin{array}{c} \Lambda_i H_x= H_{y,i} \Gamma_i, \Lambda_i h_x \le h_{y,i} + H_{y,i} \beta_i, i=1,\cdots,N, \\ 
 (\bar{x},X)=\displaystyle \sum_{i=1}^N (\bar{y}_i-Y_i \beta_i,Y_i \Gamma_i). \end{array}$ 
 & Proposition \ref{result_minkowski_right} & \text{\sffamily X}  \\

\hline 
  \rowcolor{black!10}
  $\bar{x} + X \mathbb{P}_x \subseteq \convh(\{\mathbb{P}_{y,i}\}_{i=1,\cdots,N})$ 
  &
  $\begin{array}{c}
\Lambda_i H_x= H_{y,i} \Gamma_i, \Lambda_i h_x \le \lambda_i h_{y,i} + H_{y,i} \beta_i, i=1,\cdots,N,  \\
 (1,\bar{x},X)=\displaystyle \sum_{i=1}^N (\lambda_i,-\beta_i,\Gamma_i). 
 \end{array}
 $
 & 
 Proposition \ref{prop_convexhull} & \text{\sffamily X}  \\

 \hline
 \rowcolor{black!10}
 $\bar{x} + X \mathbb{P}_x \subseteq \convh(\{\bar{y}_i+ Y_i \mathbb{P}_{y,i}\}_{i=1,\cdots,N})$ & 
 $
 \begin{array}{c}
\Lambda_i H_x= H_{y,i} \Gamma_i, \Lambda_i h_x \le \lambda_i h_{y,i} + H_{y,i} \beta_i, \lambda_i \ge 0,\\
 i=1,\cdots,N, (1,\bar{x},X)=\displaystyle \sum_{i=1}^N (\lambda_i,\lambda_i \bar{y}_i-Y_i \beta_i,Y_i\Gamma_i). 
 \end{array}
 $
 & Corollary \ref{prop_convexhull_AH} & \text{\sffamily X}  \\

 \hline
 \rowcolor{black!10}
 $\displaystyle \bigvee_{i=1}^N ( \bar{x} + X \mathbb{P}_x \subseteq \mathbb{P}_{y,i})$ &
  $
 \begin{array}{c}
\Lambda_i H_x= H_{y,i} X_i, \Lambda_i h_x \le \delta_i h_{y,i} - H_{y,i} \beta_i, \delta_i \in \{0,1\},  \\
 i=1,\cdots,N, (1,\bar{x},X)=\displaystyle \sum_{i=1}^N (\delta_i,\beta_i,X_i). 
 \end{array}
 $
 & Proposition \ref{prop_disjuntcive} & \checkmark  \\

 \hline 
 \rowcolor{black!10}
$ \langle \bar{x},X \rangle \subseteq \convh(\bigcup_{i=1,\cdots,N}\{ \langle \bar{y}_i,Y_i \rangle \})$ &

$\begin{array}{c}
\|(\Gamma_i|\beta_i)\|_{\infty} \le \lambda_i, \lambda_i \ge 0, i=1,\cdots,N, \\
\displaystyle (1,\bar{x},X)=\sum_{i=1}^N (\lambda_i,Y_i \beta_i + \lambda_i \bar{y}_i, Y_i \Gamma_i).
\end{array}
$
& Corollary \ref{cor_zono_hull} & \text{\sffamily X} \\

 \hline 
 \rowcolor{black!10}
$ \displaystyle \bigvee_{i=1}^N (\langle \bar{x},X \rangle \subseteq \langle \bar{y}_i,Y_i \rangle \})
$ &

$
\begin{array}{c}
\|(\Gamma_i|\beta_i)\|_{\infty} \le \delta_i, \delta_i \in \{0,1\}, i=1,\cdots,N, \\
\displaystyle (1,\bar{x},X)=\sum_{i=1}^N (\delta_i,Y_i \beta_i + \delta_i \bar{y}_i, Y_i \Gamma_i).
\end{array}
$
& Corollary \ref{cor_zono_dis}& \text{\sffamily X} \\
\hline

\end{tabular}
}
\end{table*}

\section{Applications}
In this section, we demonstrate the usefulness of the results in this paper on applications in computational convexity and problems encountered in formal methods approach to control theory.

\subsection{Polytopic Hausdorff Distance}
\label{sec_hausdorff}
Given a metric $d: \mathbb{R}^n \times \mathbb{R}^n \rightarrow \mathbb{R}$, the Hausdorff distance provides a metric for subsets of $\mathbb{R}^n$, which is defined as follows:
\begin{equation}
\label{eq_hausdorff}
d_H(\mathbb{S}_1,\mathbb{S}_2) := \max \{ D_{12}, D_{21} \},
\end{equation}
where  $D_{12}$, $D_{21}$ are given as:
\begin{equation}
\label{eq_directed}
\begin{array}{c}
D_{12}:= \sup_{s_2 \in \mathbb{S}_2} \inf_{s_1 \in \mathbb{S}_1} d(s_1,s_2), \\
D_{21}:=  \sup_{s_1 \in \mathbb{S}_1} \inf_{s_2 \in \mathbb{S}_2} d(s_1,s_2).
\end{array}
\end{equation}
$D_{12}$ and $D_{21}$ are known as directed distances, which do not satisfy metric properties as it is possible to have $D_{21} \neq D_{12}$. Note that $D_{ij}=0$ if and only if $\mathbb{S}_i \subseteq \mathbb{S}_j$, $(i,j) \in \{(1,2),(2,1)\}$. For compact sets, the $\sup$ and $\inf$ in \eqref{eq_directed} can be replaced by $\max$ and $\min$, respectively, and $d_H(\mathbb{S}_1,\mathbb{S}_2)=0$ if and only if $\mathbb{S}_1=\mathbb{S}_2$. 

The Hausdorff distance has applications in computer vision \cite{huttenlocher1993comparing} and set-valued control \cite{runnger2017computing}. To the best of our knowledge, all existing algorithms for the exact computation of Hausdorff distances between polytopes rely on vertex enumeration. While this is not a significant problem in computer vision, where the dimensions are not greater than 2 or 3, it introduces computational bottlenecks in high dimensional control applications. Using Proposition \ref{result_minkowski_right}, we provide a linear programming approach to compute an upper-bound for the Hausdorff distance between two polytopes.  

It is straightforward to show that, if $\mathbb{S}_1, \mathbb{S}_2$ are compact sets, \eqref{eq_hausdorff} can be written as: 
\begin{equation}
\label{eq_hausdorff_compact}
d_H(\mathbb{S}_1,\mathbb{S}_2) = \max \{ \min_{\mathbb{P}_1 \subseteq \mathbb{P}_2 \oplus d_1 \mathbb{P}_{\ball}} d_1 ,  \min_{\mathbb{P}_2 \subseteq \mathbb{P}_1 \oplus d_2 \mathbb{P}_{\ball}} d_2  \},
\end{equation}
where $\mathbb{P}_{\ball}$  is the unit ball in $\mathbb{R}^n$ corresponding to the underlying norm. For 1-norm and $\infty$-norm, this ball is a polytope. For Euclidean norm and other norms, the ball has to be approximated by polytopes. There are methods to  approximate Euclidean norm ball by zonotopes. The approximation can be made arbitrarily precise by increasing the zonotope order \cite{bourgain1993approximating}. 
We denote the polytope corresponding to the unit ball in $\mathbb{R}^n$ by $\mathbb{P}_{\ball}=\{ x \in \mathbb{R}^n | H_{\ball} x \le h_{\ball} \}$.  We use Proposition \ref{result_minkowski_right} to convert \eqref{eq_hausdorff_compact} into a linear program. The norm used in the examples in this paper is $\infty$-norm.


\begin{proposition}
\label{prop_hausdorff}
Given two AH-polytopes $\mathbb{X}_1,\mathbb{X}_2 \subset \mathbb{R}^n$, where $\mathbb{X}_i=\bar{x}_i+X_i \mathbb{P}_i, \mathbb{P}_i=\{z \in \mathbb{R}^{n_i} | H_i z \le h_i\}, i=1,2, X_i \in \mathbb{R}^{n \times {n_i}}, \bar{x}_i \in \mathbb{R}^n$, the following linear program provides an upper-bound for their Hausdorff distance:
\begin{equation}
\label{eq_lp_hausdorff}
\begin{array}{ll}
\min & D  \\ 
\text{subject to} & \Lambda_1 H_1 = H_2 \Gamma_1, \Lambda_2 H_1 = [I,-I] \Gamma_2, \\
& \Lambda_1 h_1 \le h_2 + H_2 \beta_1 \\
& \Lambda_2 h_1 \le D h_{\ball} + H_{\ball} \beta_2 \\
& \Lambda_3 H_2 = H_1 \Gamma_3, \Lambda_3 H_4 = H_{\ball} \Gamma_4, \\
& \Lambda_3 h_2 \le h_2 + H_1 \beta_3 \\
& \Lambda_4 h_2 \le D h_{\ball} + H_{\ball} \beta_4 \\
& \Lambda_i \ge 0, i=1,2,3,4, \\
& \bar{x}_2 - X_2 \beta_1 - \beta_2 = \bar{x}_1, X_2 \Gamma_1 + \Gamma_2 =X_1, \\
& \bar{x}_1 - X_1 \beta_3 - \beta_4 = \bar{x}_2, X_1 \Gamma_3 + \Gamma_4 =X_2. \\
 \end{array}
\end{equation} 
\end{proposition}
\begin{proof}
Using \eqref{eq_hausdorff_compact}, we have
\begin{equation}
\label{eq_haus}
\begin{array}{lll}
d_H(\mathbb{P}_1,\mathbb{P}_2) =& \min & D  \\ 
& \text{subject to} & \mathbb{X}_1 \subseteq \mathbb{X}_2 \oplus D \mathbb{P}_{\ball}, \\
 &&  \mathbb{X}_2 \subseteq \mathbb{X}_1 \oplus D \mathbb{P}_{\ball}, \\
 \end{array}
\end{equation} 
where $D \mathbb{P}_{\ball}$ can be written as $\{x \in \mathbb{R}^n | H_{\ball} x \le D h_{\ball} \}$. Using Proposition \ref{result_minkowski_right} and some algebraic manipulation, we arrive at \eqref{eq_lp_hausdorff}. Since Proposition \ref{result_minkowski_right} is a sufficient condition, the optimal $D$ in \eqref{eq_lp_hausdorff} satisfies the constraints in \eqref{eq_haus} is an upper bound for $d_H(\mathbb{X}_1,\mathbb{X}_2)$.   
\end{proof}

The tightness of the upper bound provided by Proposition \ref{prop_hausdorff} is as good as the necessity of Proposition \ref{result_minkowski_right}. Trivial cases are handled in a sensible way. For example, $\mathbb{X}_1=\mathbb{X}_2$ if and only if \eqref{eq_lp_hausdorff} returns 0. Note that if lines corresponding to $\mathbb{X}_2 \subseteq \mathbb{X}_1 \oplus D \mathbb{P}_{\ball}$ in \eqref{eq_lp_hausdorff} are removed, we arrive in an upper-bound for the directed distance $D_{12}$, which becomes zero if and only if $\mathbb{X}_1 \subseteq \mathbb{X}_2$. 

\begin{example}
\label{examples_zono_distance}
Consider the zonotopes in Example \ref{example_zonotopes}. If the last column of $Y$ is dropped (see Fig. \ref{fig_zonotopes} [Right]). Using Proposition \ref{prop_hausdorff} (which takes a simpler form for zonotopes but we omit its derivation for brevity), we obtain the following upper bounds $D_{12} \le 2$ and $D_{21} \le3$, so $d_H(\mathbb{Z}^*_{y},\mathbb{Z}_x) \le 3$. Augmented zonotopes $\mathbb{Z}_x \oplus D_{12} \mathbb{B}_2$ and $\mathbb{Z}^*_{y} \oplus D_{21} \mathbb{B}_2$ are shown in Fig. \ref{fig_zonoko}.

\begin{figure}
\centering
\includegraphics[width=0.22\textwidth]{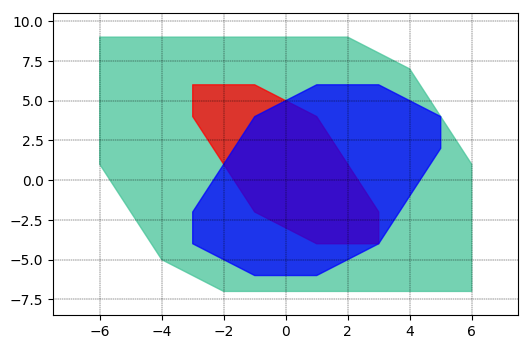}
\includegraphics[width=0.22\textwidth]{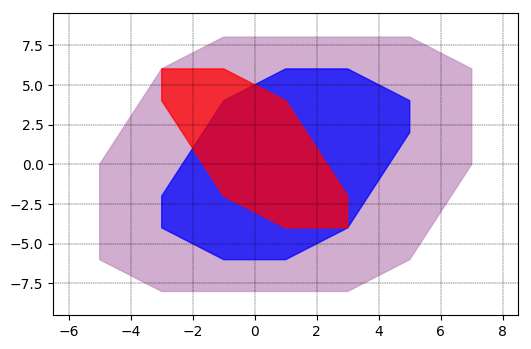}
\label{fig_zonoko}
\caption{Example \ref{examples_zono_distance}: Computing Hausdorff distance between zonotopes: [left] minimal ball added to $\mathbb{Z}_x$ (red) to contain $\mathbb{Z}^*_{y}$ (blue) [Right] minimal ball added to $\mathbb{Z}^*_{y}$ to contain $\mathbb{Z}_{x}$. The Hausdorff distance is upper bounded by $\max(2,3)=3$.}
\end{figure}
\end{example}

\begin{remark}
\label{remark_hausdorff}
An alternative characterization of the Hausdorff distance can be made using support functions of convex sets as follows: 
\begin{equation}
\label{eq_support_hausdorff}
d_H(\mathbb{C}_1,\mathbb{C}_2) = \max_{c \in \mathbb{P}_{\ball}} | \max_{x_1 \in \mathbb{C}_1} c'x_1 - \max_{x_2 \in \mathbb{C}_2} c'x_2 |,
\end{equation}
where $\mathbb{C}_1$ and $\mathbb{C}_2$ are compact convex sets. Unfortunately, $| \max_{x_1 \in \mathbb{C}_1} c'x_1 - \max_{x_2 \in \mathbb{C}_2} c'x_2 |$ is not a concave function. A practical way to obtain an estimate of $d_H(\mathbb{C}_1,\mathbb{C}_2)$ is to compute $| \max_{x_1 \in \mathbb{C}_1} c'x_1 - \max_{x_2 \in \mathbb{C}_2} c'x_2 |$ for a large number of values of $c$ sampled from the boundary of $\mathbb{P}_{\ball}$, and take the maximum. This provides a lower bound for the Hausdorff distance, and was implemented in \cite{yang2018comparison} for computing the Hausdorff distance between zonotopes. Proposition \ref{prop_hausdorff} provides an upper bound, which is relatively tight in some applications as suggested by the conservativeness of Theorem \ref{thm_zono}.
\end{remark}

\subsection{Zonotope Order Reduction}
As mentioned earlier, Minkowski sums of zonotopes have a very convenient representation. When a large number of Minkowski sums is required, the zonotope order can grow too large for practical considerations. For example, the order of a zonotope characterizing the reachable set of a discrete-time dynamical system subject to zonotopic additive disturbances constantly increases over time. Therefore, it is desired to approximate a zonotope by another zonotope of lower order. In many forward reachability problems, outer-approximation preserves safety properties. In other words, the lower order zonotope must contain the original zonotope. Therefore, zonotope order reduction subject to outer-approximation has got special attention in the controls literature \cite{girard2005reachability,kopetzki2017methods,yang2018comparison}. Proposed heuristics include methods based on interval-hulls, principal component analysis (PCA), K-means clustering of generator columns, and singular value decomposition (SVD). Characterization of the approximation tightness is non-trivial. The authors in \cite{kopetzki2017methods} used volume ratio, which its computation in high dimensions is problematic itself. The authors in \cite{yang2018comparison} used Hausdorff distance, but their method provided a lower-bound by sampling vectors from the boundary of the unit ball, as described earlier in Remark \ref{remark_hausdorff}. Here we use our results to formulate the zonotope order reduction problem, both outer-approximation and inner-approximation, as an optimization problem, where the cost is an upper-bound for the Hausdorff distance between the original and reduced zonotopes.     

\begin{proposition}[Optimal Zonotope Order Reduction - Outer Approximation]
\label{zono_outer}
Given zonotope $\mathbb{Z}=\langle \bar{x},X \rangle$, with $X \in \mathbb{R}^{n \times n_x}$, let the best reduced zonotope $\mathbb{Z}_{red}=\langle \bar{x},X^*_{red} \rangle$ be such that $\mathbb{Z} \subseteq \mathbb{Z}_{red}$, where $X^*_{red} \in \mathbb{R}^{n \times n_x^{\reduced}}$, $n_x^{\reduced} \le n_x$, be given by the optimal solution to the following optimization problem: 
\begin{equation}
\label{eq_zono_outer}
\begin{array}{lll}
X^*_{\reduced} = & \arg \min & \delta \\
& \text{subject to} & X=X_{\reduced} \Gamma_0, \\
&  & X_{\reduced}=X \Gamma_1 + \Delta, \\
&  & \| [\Gamma_0,\Gamma_1] \|_\infty \le 1, \\
&  & \| \Delta \|_\infty \le \delta. \\
\end{array}
\end{equation}
Then we have $d_{H}(\mathbb{Z},\mathbb{Z}_{\reduced}) \le \delta^*$, where $\delta^*$ is the optimal objective in \eqref{eq_zono_outer}. 
\end{proposition}
\begin{proof}
The proof is based on Proposition \ref{result_minkowski_right} and Proposition \ref{prop_hausdorff}.  From Theorem \ref{thm_zono}, we have $\langle 0,X_{red} \rangle \subseteq \langle 0,X \rangle \oplus \delta \langle 0,I \rangle $ if $X_{\reduced}=(X,\delta I)[\Gamma_1,\Gamma_2]$, where $[\Gamma_0,\Gamma_1] \|_\infty \le 1$. Define new variable $\Delta := \delta \Gamma_2$, and we have $\|\Delta \|_\infty \le \delta$, and the rest of the proof is straightforward.  
\end{proof}
\begin{corollary}[Optimal Zonotope Order Reduction - Inner Approximation]
\label{zono_inner}
Given zonotope $\mathbb{Z}=\langle \bar{x},X \rangle$, with $G \in \mathbb{R}^{n \times n_x}$, let the best reduced zonotope $\mathbb{Z}_{red}=\langle \bar{x},X^*_{red} \rangle$ be such that $\mathbb{Z}_{red} \subseteq \mathbb{Z}$, where $X^*_{red} \in \mathbb{R}^{n \times n_x^{\reduced}}$, $n_x^{\reduced} \le n_x$, is given by the optimal solution to the following optimization problem: 
\begin{equation}
\label{eq_zono_inner}
\begin{array}{lll}
X^*_{\reduced} = & \arg \min & \delta \\
& \text{subject to} & X_{\reduced} = X \Gamma_0, \\
&  & X=X_{\reduced} \Gamma_1 + \Delta, \\
&  & \| [\Gamma_0,\Gamma_1] \|_\infty \le 1, \\
&  & \| \Delta \|_\infty \le \delta.\\
\end{array}
\end{equation}
Then we have $d_{H}(\mathbb{Z},\mathbb{Z}_{red}) \le \delta^*$, where $\delta^*$ is the optimal objective in \eqref{eq_zono_inner}. 
\end{corollary}

One of the main advantages of Proposition \eqref{eq_zono_outer}  and Corollary \eqref{eq_zono_inner} is that the order of the reduced zonotope is defined arbitrarily by the user. This is contrast to a many available methods in \cite{kopetzki2017methods} in which reduced zonotopes of only specific orders can be considered. Both \eqref{eq_zono_outer}  and \eqref{eq_zono_inner} have a bilinear matrix equality term, which makes the optimization problem non-convex. Our solution is based on bilinear alternations, which requires an initial guess. The methods described in \cite{kopetzki2017methods} can be used for an initial guess. Note that finding an initial feasible solution to both \eqref{eq_zono_outer} and \eqref{eq_zono_inner} is straightforward - let the zonotope to be very large (respectively, very small) for outer (respectively, inner) approximation - however finding a ``good" initialization can be difficult.

\begin{example}
\label{example_zono_red}
\begin{figure*}
\centering
\includegraphics[width=0.18\textwidth]{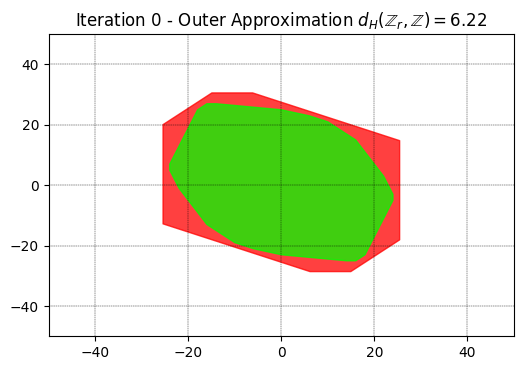}
\includegraphics[width=0.18\textwidth]{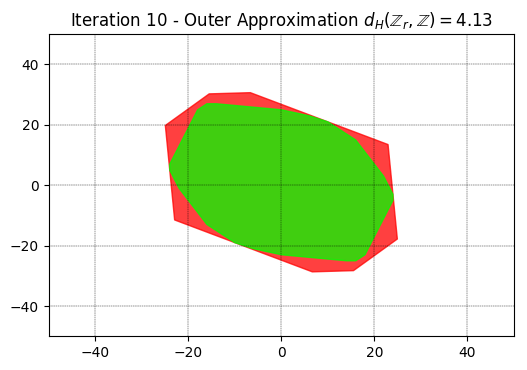}
\includegraphics[width=0.18\textwidth]{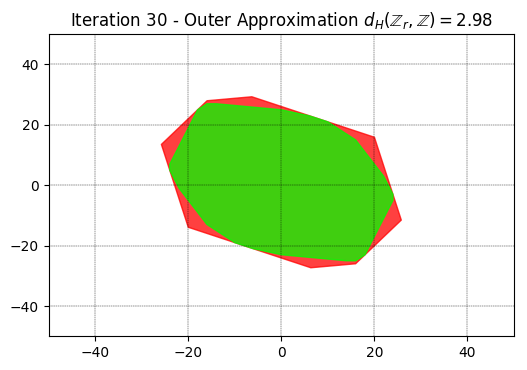}
\includegraphics[width=0.18\textwidth]{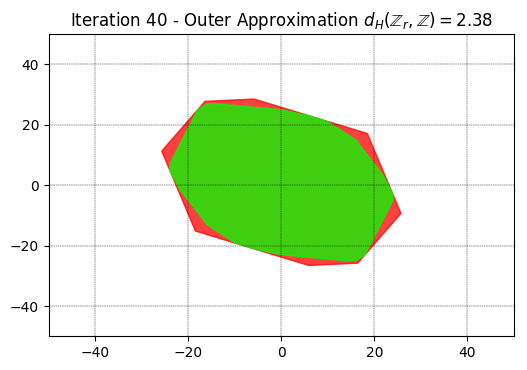}
\includegraphics[width=0.18\textwidth]{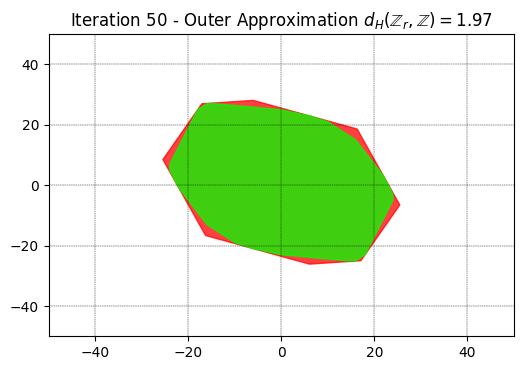}
\includegraphics[width=0.18\textwidth]{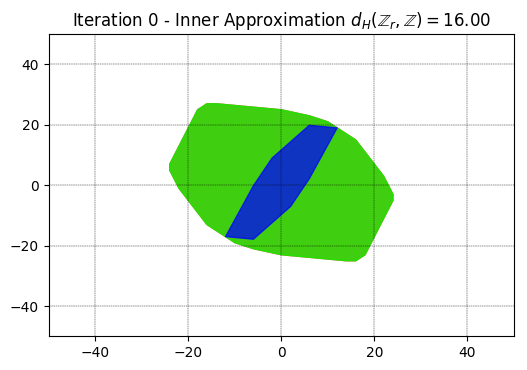}
\includegraphics[width=0.18\textwidth]{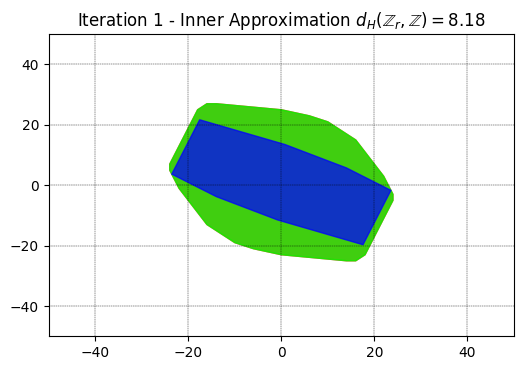}
\includegraphics[width=0.18\textwidth]{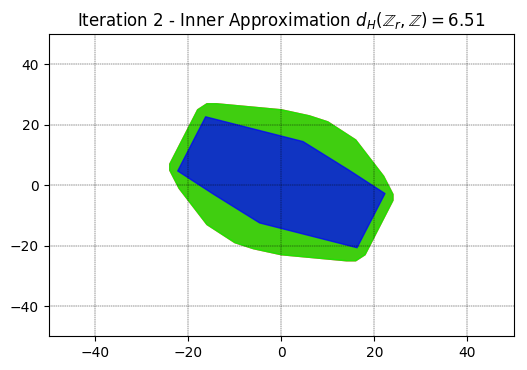}
\includegraphics[width=0.18\textwidth]{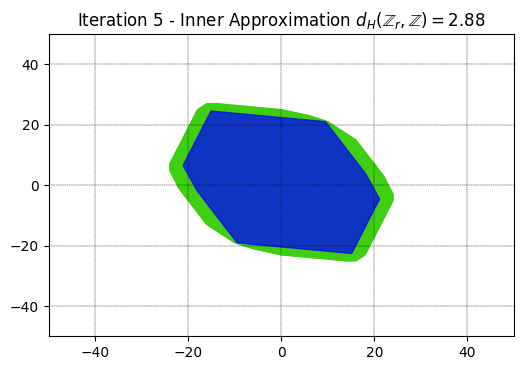}
\includegraphics[width=0.18\textwidth]{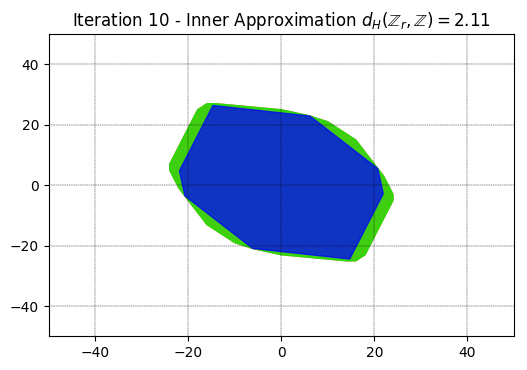}
\label{fig_zonotope_order}
\caption{Example \ref{example_zono_red}: Optimization-based iterative zonotope order reduction by a factor of 3: [Top] Iterations of projected gradient decent through SLP for order reduction subject to outer-approximation [Bottom] Iterations of bilinear alternation through SLP for zonotope order reduction subject to inner-approximation.}
\end{figure*} 

We consider a 2D zonotope with order 6, i.e. 12 generators. We aim to find both the best inner and outer approximations of order 2.  In order to solve \eqref{eq_zono_outer}, we initialized $X^0_{\reduced}$ by a random matrix $R$. Then we found a diagonal matrix $D$ with minimal sum of the squares of the diagonal terms such that $X=R D \Gamma_0$ is feasible for some $\|\Gamma_0\|_\infty$. Therefore, $X^0_{\reduced}=RD$ becomes a tight initial guess. We multiplied $X^0_{\reduced}$ by $1.05$ to break tightness a little, and then solved the nonlinear optimization in \eqref{eq_zono_outer} by finding the optimal changes in the matrix $X_{\reduced}$. Each iteration is a linear program where a maximal change of $0.1$ in each entry in $X_{\reduced}$ is allowed, and the following approximation is used to break the bilinearity:
$$
X^{i+1}_{\reduced} \Gamma^{i+1}_0 \approx  (X^{i+1}_{\reduced}-X^{i}_{\reduced}) \Gamma^{i}_0 + X^{i}_{\reduced} \Gamma^{i+1}_0,
$$
where $i \in \mathbb{N}$ is the index of iterations.  
 The iterations are shown in Fig. \ref{fig_zonotope_order} [Top]. Note that since $X_{\reduced}$ has smaller size than $X$, keeping $\Gamma_0$ fixed and solving for $X_{\reduced}  $ and vice versa is not applicable for \eqref{eq_zono_outer} - given $\Gamma_0$, the solution for $X_{\reduced}$ is unique. Therefore, the method resembles to projected gradient decent. On the other hand, the case for \eqref{eq_zono_inner} is simpler as we use bilinear alternations between $X_{\reduced}$ and $\Gamma_1$ to find the inner-approximation. The convergence is often faster than the projected gradient decent. Iterations are shown in Fig. \ref{fig_zonotope_order} [Bottom]. The scripts of both examples are available in the \texttt{test} folder in our python package \texttt{pypolycontain}.     
\end{example}

\subsection{Orthogonal Projections}

Computing orthogonal projections of polytopes is a central problem in many applications such as computing feasible regions of model predictive controllers. The exact computation of projections requires variable elimination, which is a costly procedure. Vertex-based projections are more convenient to implement but do not scale well in high dimensions. Moreover, it is often preferred to have the H-polytope rather than a V-polytope of the projection. Not only the projection procedure  is computationally intense, but the number of hyperplanes in the projected polytope itself may be inevitably too large, in particular when a high dimensional polytopes are projected into low dimensional spaces.

We use the results of this paper to introduce an output-sensitive inner-approximation alternative to orthogonal projection that is based on a single optimization problem. Consider the set $\mathbb{F}=\{ (x,u) \in \mathbb{R}^{n+m} | H x + Fu \le g\}$. The projection into $x$-space is given by $\mathbb{F}_{\proj}=\{x | \exists u \in \mathbb{R}^m, (x,u) \in \mathbb{F}\}$. 
We desire to find $H_x$ and $h_x$ in $\mathbb{X}= \{ x \in \mathbb{R}^n | H_x x \le h_x \}$ such that $\mathbb{X} \subseteq \mathbb{F}_{\proj}$ and $d_H(\mathbb{X},\mathbb{F}_{\proj})$ is minimized. The choice of the number of rows in $H_x$ is made by the user - we expect $d_H(\mathbb{X},\mathbb{F}_{\proj})$ to decrease with the number of rows. While there has been theoretical results on the quality of approximating projections of high dimensional polytopes by polytopes with controlled number of facets \cite{barvinok2006computational}, to the best of our knowledge, the following approach is unique in the respect that it is optimization-based, and provides a guaranteed upper bound on the Hausdorff distance between the approximated projected polytope and the actual one, which does not need to be computed.   

We parameterize $\mathbb{X}$ by $\bar{x} + \{ x \in \mathbb{R}^n | H_x x \le \underbar{1} \}$. Note that by design, $\mathbb{X}$ contains $\bar{x}$. The optimal values in $H_x$ are given by the following optimization problem:  
\begin{equation}
\begin{array}{rl}
H_x = & \arg \min \epsilon \\
\text{subject to} & \mathbb{F}=\{ (x,u) \in \mathbb{R}^{n+Nm} | H x + Fu \le g\} \\
& \mathbb{X}= \{ x \in \mathbb{R}^n | H_x \le \underbar{1} \}, \\
& \bar{x} + \mathbb{X} \subseteq (I,0) \mathbb{F} \\
& (I,0) \mathbb{F} \subseteq \bar{x} + \mathbb{X} \oplus \epsilon \mathbb{B} 
\end{array}
\end{equation}
Using the containment encoding framework provided in this paper and some algebraic manipulation, we have the following result. 


\begin{proposition}
\label{prop_projection} 
Given $\mathbb{F}=\{ (x,u) \in \mathbb{R}^{n+m} | H x + Fu \le g\}$, let $\mathbb{F}_{\proj}=\{x | \exists u \in \mathbb{R}^m, (x,u) \in \mathbb{F}\}$. Consider the following optimization problem:
\begin{equation}
\begin{array}{rl}
H_x = & \arg \min {\epsilon}, \\
\text{subject to} & {\Lambda_0 H_x} = H + F {\Gamma} \\
 & {\Lambda_0} \underbar{1} \le g - H \bar{x} + F {\beta_u}, \\
& {\Lambda_1} H = {H_x X_1}, \\
& {\Lambda_1} F = {H_x X_2}, \\
& {\Lambda_2} H = H_B - H_{\ball} {X_1}, \\
& {\Lambda_2} F = - H_{\ball} {X_2}, \\
& {\Lambda_1} g \le h_{\ball} - {H_x} {\beta_x}, \\
& {\Lambda_2} g \le {\epsilon} h_{\ball} - H_{\ball} \bar{x} + H_{\ball} {\beta_x}, \\
&  {\Lambda_0, \Lambda_1, \Lambda_2} \ge 0. \\
\end{array}
\label{eq_projection}
\end{equation}
Let $H_x$ be a feasible solution with cost $\epsilon$, and $\mathbb{X}=\bar{x}+ \{x \in \mathbb{R}^n | H_x x \le \underbar{1} \}$. Then we have $\mathbb{X} \subseteq \mathbb{F}_{\proj}$ and $d_H (\mathbb{X}, \mathbb{F}_{\proj}) \le \epsilon.$ 
\end{proposition}

The optimization problem \eqref{eq_projection} has some bilinear terms, therefore it is difficult to solve it to global optimality. Nevertheless, we can use local methods to find suboptimal solutions. 

\begin{example}
\label{example_ortho}
\begin{figure*}
\centering
\includegraphics[width=0.18\textwidth]{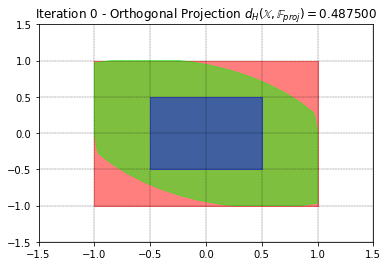}
\includegraphics[width=0.18\textwidth]{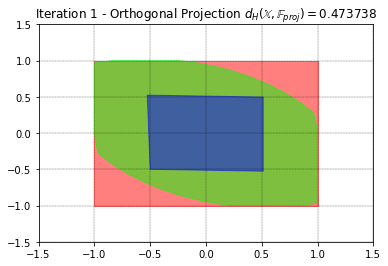}
\includegraphics[width=0.18\textwidth]{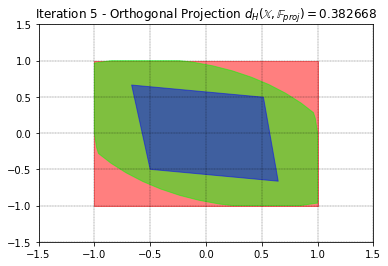}
\includegraphics[width=0.18\textwidth]{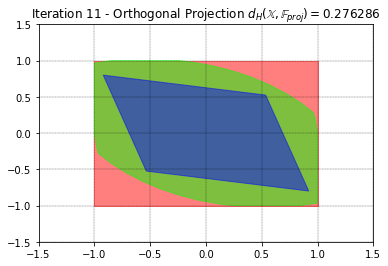}
\includegraphics[width=0.18\textwidth]{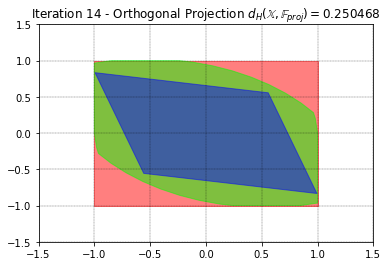}
\includegraphics[width=0.18\textwidth]{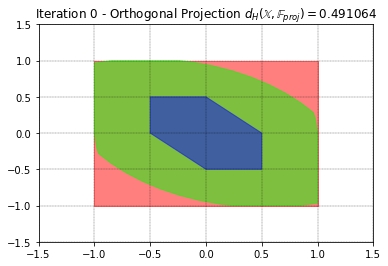}
\includegraphics[width=0.18\textwidth]{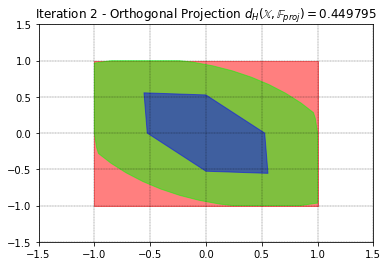}
\includegraphics[width=0.18\textwidth]{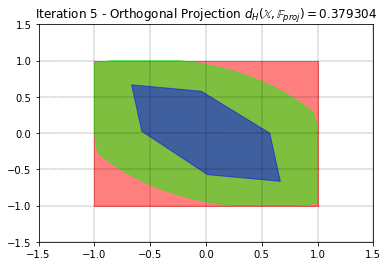}
\includegraphics[width=0.18\textwidth]{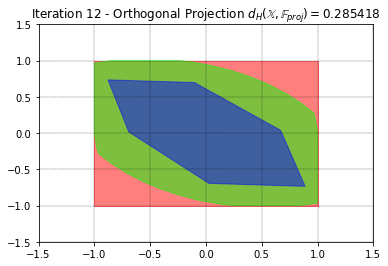}
\includegraphics[width=0.18\textwidth]{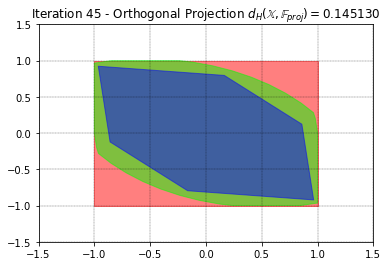}
\label{fig_projection}
\caption{Example \ref{example_ortho}: Optimization-based iterative orthogonal projection from $\mathbb{R}^{22}$ to $\mathbb{R}^{2}$ to obtain the feasible set of a model predictive controller. The state box constraint is shown in red, the actual $\mathbb{F}_{\proj}$ is shown in green, and the inner-approximation $\mathbb{X}$ is shown in blue. [Top] 4 hyperplanes [Bottom] 6 hyperplanes. The green polytope is the actual projection obtained by Fourier-Motzkin elimination of $20$ variables. }
\end{figure*} 
In this example, we wish to compute an inner-approximation of the feasible set of a model predictive controller. Consider the linear system $x_{t+1}=Ax_t+Bu_t$, where 
\begin{equation*}
A=\left( \begin{array}{cc}
1 & 0.1 \\
-0.1 & 1 
\end{array}
\right),
B=\left( \begin{array}{c}
0 \\
0.1 
\end{array}
\right).
\end{equation*}
We impose hard constraints $x \in [-1,1]^2$ and $u \in [-1,1]$. We wish to compute the set of states that can be steered into the origin in $N=20$ steps, while satisfying the box constraints. We have:
$$
\begin{array}{c}
\mathbb{F}= \{ x_0, u_0,\cdots,u_{N-1} | x_N =0 \}, \\
\mathbb{F}_{\proj}= \{ x_0 \big| \exists u_0,\cdots,u_{N-1}, \st x_N =0 \}, 
\end{array}
$$ 
where 
$$
x_n=A^n x_0 + \sum_{\tau=0}^n A^{n-\tau-1}Bu_\tau.
$$
We need to find $\mathbb{F}_{proj}$, which is provided by projecting $N+2=22$ dimensional polytope of joint state and control sequence space into 2-dimensional state-space. By writing the constraints described above, the number of hyperplanes in $\mathbb{F}$ is 128 (some may be redundant). 

We may use two methods to compute this projection. The first is the exact and is given by Fourier-Motzkin elimination method. Its computation is costly as $20$ variables are eliminated, and at each iteration, many linear programs are required to remove redundant hyperplanes. The method returns the exact $\mathbb{F}_{\proj}$ with 28 irreducible hyperplanes. 

As an alternative, we use the method described in this paper in Proposition \ref{prop_projection}. We consider two initializations: $H_x$ with $4$ hyperplanes (a box), and $H_x$ with $6$ hyperplanes. We consider maximum step size of $0.05$ in each entry of matrix variables. We let $\bar{x}=0$. Snapshots of the gradient decent iterations are shown in Fig. \ref{fig_projection}. It is observed that we are able to closely inner-approximate the MPC feasible set using user-defined number of hyperplanes, while not encountering any potential exponential blow up due to vertex elimination or the number of required hyperplanes. Note that the reported Hausdorff distances are the values of $\epsilon$ in \eqref{eq_projection}, which are  guaranteed to be upper-bounds and are computed without explicitly knowing $\mathbb{F}_{\proj}$.   

\endproof
\end{example}

\subsection{Verification and Control of Hybrid Systems}
\label{sec_control}

In this section, we study a particular application of our results to formal synthesis of controllers for hybrid systems. The method is based on our framework in \cite{sadraddini2018sampling}, where polytope-to-polytope control strategies were used to design feedback strategies for piecewise affine systems. The task is reaching a goal region, which itself is given as a union of polytopes in the state-space, while respecting state and control constraints. The polytopes form a tree that grows backward from the goal using a sampling-based heuristic similar to rapidly-exploring random trees (RRT) \cite{lavalle2006planning}. The idea is also closely related to sampling-based feedback motion planning using LQR-trees \cite{tedrake2010lqr} - but tailored for piecewise affine systems. The central technique is computing \emph{polytopic trajectories} that are characterized by polytopes $\mathbb{P}_t, t=0,1,\cdots,T$, in the state-space, where $\mathbb{P}_t$ is mapped to $\mathbb{P}_{t+1}$ with an appropriate control law. The nodes of the tree are polytopes, and edges represent available one-step controlled transitions.  

The polytopes are represented in AH-polytope form, where 
\begin{equation}
\label{eq_parapolytope}
\mathbb{P}_t = \bar{x}_t + G_t \mathbb{P}_b,
\end{equation} 
$\mathbb{P}_{b} \subset \mathbb{R}^q$ is a user-defined base polytope, and $\bar{x}_t \in \mathbb{R}^n, G_t \in \mathbb{R}^{n \times q}$ define the affine transformation. Note that when $\mathbb{P}_b$ is chosen as a box, all polytopes become zonotopes. Let the affine dynamics be ${x}_{t+1}=A_t x_t + B_t u_t + c_t$. Consider the parameterized control strategy 
\begin{equation}
\label{eq_u}
u_t=\bar{u}_t+\theta_t\zeta, x_t=\bar{x}_t+G_t \zeta,
\end{equation} 
where $\zeta \in \mathbb{R}^q$ is an implicit variable. Unless $G_t$ is invertible, the map from $x_t$ to $u_t$ is given by a linear/quadratic program, with an ad-hoc cost function (for example, the norm of $u_t$). Note that multiple $u_t$ values for state $x_t$ may satisfy \eqref{eq_u}. Using control law \eqref{eq_u}, the evolution of $\mathbb{P}_t$ satisfies the following relation:
\begin{equation}
\label{eq_parapolytope_next}
\mathbb{P}_{t+1} = A_t\bar{x}_t+B_t\bar{u}_t + c_t + (A_tG_t+B_t \theta_t) \mathbb{P}_b
\end{equation} 
or, equivalently: 
\begin{equation}
\label{eq_parapolytope_next}
\bar{x}_{t+1}= A_t\bar{x}_t+B_t\bar{u}_t + c_t, G_{t+1} = A_tG_t+B_t \theta_t.
\end{equation}
Therefore, we have linear encodings for the polytopic trajectory. We can also consider mixed-integer formulations to encode hybrid relations between $(A_t,B_t,c_t)$ and $(x_t,u_t)$, but the details are omitted here. The full algorithm and its theoretical guarantees are reported in \cite{sadraddini2018sampling}. Here we only provide the essential details related to the contributions of this paper.

Using the results of this paper, we improve a crucial computational part of the algorithm in \cite{sadraddini2018sampling}. When adding a branch to the tree, we design a polytopic trajectory such that the final polytope is contained within one of the polytopes in the tree, i.e. an instance of polytope containment problem. In \cite{sadraddini2018sampling}, we explored efforts to compute the H-polytope form of these AH-polytopes, either using Fourier-Motzkin elimination which was very slow and numerically unstable for our applications, or approximating the transformation matrix by a left-invertible one which also caused numerical issues as fine approximations led to ill-conditioned matrices. Here we use our results on AH-polytope in AH-polytope containment, in particular zonotope containment, to present an alternative approach that does not require H-polytope forms of the polytopes in the tree.

\begin{example}
\label{example_tobia}
We adopt example 1 from \cite{marcucci2017approximate}, which was also studied in \cite{sadraddini2018sampling}. The model represents an inverted pendulum with a spring-loaded wall on one side (see Fig. \ref{example_tikz}). The control input is the external torque.   
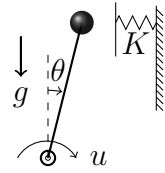
\begin{figure}[t]
\centering
\begin{tikzpicture}[scale=0.90]
\draw[thick] (0,0) -- (0.5,2);
\shade[ball color = black!99, opacity = 0.9] (0.5,2) circle (5pt);
\node[] at (0.15,1.3) {$\theta$};
\draw[thick,->] (-0.4, 1.8) -- (-0.4,1.2);
\node[] at (-0.4,0.9) {$g$};
\node[] at (1.3,1.7) {$K$};
\draw[] (1.6,0.7) -- (1.6,2.25);
\draw[] (1,1.5) -- (1,2.3);
\draw[] (1,2) -- (1.05,2.1) -- (1.15,1.9) -- (1.25,2.1) -- (1.35,1.9) -- (1.45,2.1) -- (1.55,1.9) -- (1.6,2);
\foreach \i in {1,...,15}{
\draw[] (1.6, 2.3-0.1*\i) -- (1.7,2.2-0.1*\i);}
\draw[line width=0.28mm,color=black] (0,0) circle (3pt);
\draw[line width=0.1mm] (0,0) circle (1pt);
\draw[->] (0,1) arc (90:78:1);
\draw[line width=0.1pt,dashed] (0,0) -- (0.0,1.6);
\draw[->] (-0.45,0) arc (150:30:0.5);
\node[] at (0.75,0) {$u$};
\end{tikzpicture}
\caption{Example \ref{example_tobia}: Inverted pendulum with a spring-loaded wall.}
\label{example_tikz}
\end{figure} 
The system is constrained to $|\theta| \le 0.12$, $|\dot{\theta}| \le 1$, $|u| \le 4$, and the wall is situated at $\theta=0.1$. The problem is to identify a set of states $\mathbb{X} \in \mathbb{R}^2$ and the associated control law $\mu: [-0.12,0.12] \times [-1,1] \rightarrow [-4,4]$ such that all states in $\mathbb{X}$ are steered toward origin in finite time, while respecting the constraints. It is desired that $\mathbb{X}$ is as large as possible. The dynamical system is described as a hybrid system with two modes associated with ``contact-free" and ``contact". The piecewise affine dynamics is given as:
\begin{equation*}
A_1=
\left(
\begin{array}{cc}
1 & 0.01 \\
0.1 & 1
\end{array}
\right),
A_2=
\left(
\begin{array}{cc}
1 & 0.01 \\
-9.9 & 1
\end{array}
\right),
\end{equation*} 
\begin{equation*}
B_1=B_2=
\left(
\begin{array}{c}
0 \\ 0.01
\end{array}
\right),
c_2=
\left(
\begin{array}{c}
0 \\ 0 
\end{array}
\right) ,
c_2=
\left(
\begin{array}{c}
0 \\ 1
\end{array}
\right),
\end{equation*} 
where mode 1 and 2 correspond to contact-free $\theta \le 0.1$ and contact dynamics $\theta >0.1$, respectively.

\begin{figure*}[t]
\centering
\includegraphics[width=0.24\textwidth]{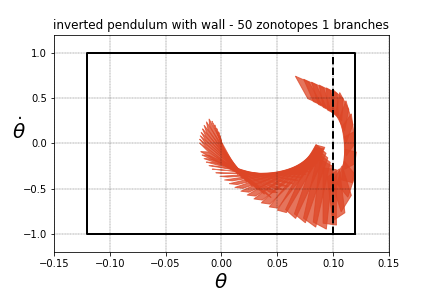}
\includegraphics[width=0.24\textwidth]{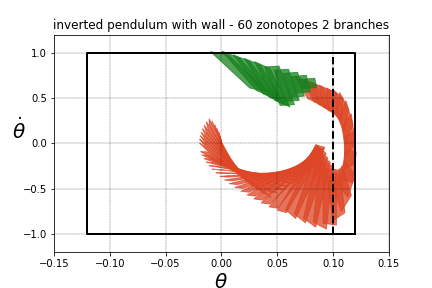}
\includegraphics[width=0.24\textwidth]{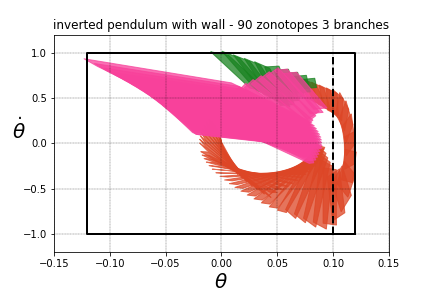}
\includegraphics[width=0.24\textwidth]{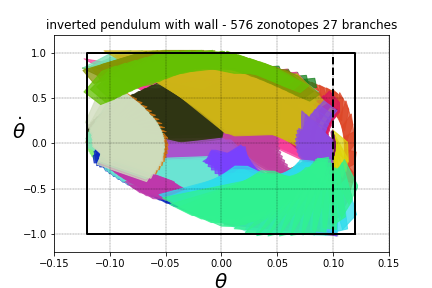}
\label{fig_polytopic_tree}
\caption{Example \ref{example_tobia}: Polytopic tree algorithm for control of a hybrid system: at each iteration, a point is sampled from the free space. Then a mixed-integer convex program is used to find a polytopic trajectory into one of the zonotopes already existing in the tree, where the results in this paper are used to encode the containment property. If such a trajectory exists, the branch is added to the tree and the iterations continue. The final result is a tree of zonotopes from which a hybrid control law steering the states to the goal is obtained.}
\end{figure*} 

The approach in \cite{marcucci2017approximate} was based on finding the feasible set of hybrid model predictive control, but the horizon was limited to $N=10$, hence the derived $\mathbb{X}$ was small. The approach in \cite{sadraddini2018sampling}, based on sampling-based polytopic trees described in this section, found larger $\mathbb{X}$. The base polytope $\mathbb{P}_b$ is chosen to be a box, hence all polytopes are zonotopes. The polytopic trajectory design is handled using a mixed-integer convex program. Here we implement the same polytopic tree algorithm with the difference we use the zonotope containment result in Theorem \ref{thm_zono} for constraining the final polytope constraint. Not only did this lead to slightly faster computations, but we also observed that the optimization solver no longer reported numerical tolerance issues.

The iterations of the algorithm are shown in Fig. \ref{fig_polytopic_tree}. Implementing the control law $\mu$ requires two cheap subroutines of finding the appropriate zonotope containing the current point and implementing the zonotope-to-zonotope control law, as opposed to the complete solution of solving a mixed-integer convex program (MICP) for the hybrid MPC with a long enough horizon. Out of 1000 points uniformly sampled from $[-0.12,0.12] \times [-1,1]$, 810 are within the feasible set of hybrid MPC with horizon $N=80$ - this set is never explicitly computed, but only we can check if $(\theta,\dot{\theta})$ belongs to it by solving a MICP problem.
After 27 branches, we observed that 796 of 810 points are inside the tree, yielding an approximate coverage of $98\%$. It was shown in \cite{sadraddini2018sampling} that full coverage is asymptotically achieved as the number of samples in the tree goes to infinity. As mentioned earlier, the full details are omitted here and the interested reader is referred to \cite{sadraddini2018sampling}. The scripts for this example are also publicly available  \footnote{\url{https://github.com/sadraddini/PWA-Control/blob/master/PWA_lib/polytree/test/inverted_pendulum_single_wall.py} }. 
 
\end{example}
%
%
%
%

\section{Conclusion and Future Work}
We studied a broad range of polytope containment problems. We provided a general sufficient conditions for containment of a polytope inside another polytope, where both polytopes are represented by affine transformations of hyperplane-represented polytopes. The significance of the method relies on the fact that the encodings do not require computing the hyperplanes of the affine transformations of the polytopes, which can be computationally prohibitive. Instead, the encoding provides a set of linear constraints with size growing linearly in the problem size.

We provided interpretations for the sufficiency, the conditions for necessity, and focused on special cases typically encountered in polytopic problems such as zonotopes, Minkowski sums, convex hulls, and disjunctive containment.  We presented the usefulness of our results on a number of applications in control theory. 

Future work will focus more deeply on applications to verification and control of hybrid systems. More specifically, we plan to use the results on disjunctive zonotope containment and zonotope in the convex hull of zonotopes to obtain more sparse polytopic trees. The convex hull of two consecutive zonotopes is a reasonable approximation for the states traversed between them in continuous time. Applications include robotic manipulation, for which fast, hybrid, correct-by-design controllers are sought. 

\bibliographystyle{IEEEtran}
\bibliography{ana_references}
\end{document}